\documentclass[10pt]{amsart}
\usepackage{latexsym,amsmath,amssymb,graphicx}
\usepackage{float}
\usepackage[bf, small]{caption}
\usepackage[all,web]{xy}
\usepackage{color}
\usepackage{hyperref}

\usepackage{xcolor}

\def\xyellowspace{%
  \sbox0{\colorbox{yellow}{\strut\ }}
  \dimen0=\wd0\relax
  \hskip0pt\cleaders\box0\hskip\dimen0\hskip0pt}

\begingroup
\catcode`\ =\active%
\gdef\makeyellowspace{\let \xyellowspace\catcode`\ =\active}%
\endgroup

\def\?#1{\colorbox{yellow}{\strut#1}}

\def\urlfont{\DeclareFontFamily{OT1}{cmtt}{\hyphenchar\font='057}
              \normalfont\ttfamily \hyphenpenalty=10000}

\DeclareFontFamily{OT1}{rsfs10}{}
\DeclareFontShape{OT1}{rsfs10}{m}{n}{ <-> rsfs10 }{}
\DeclareMathAlphabet{\mathscript}{OT1}{rsfs10}{m}{n}

\DeclareMathOperator{\im}{Im}       
\DeclareMathOperator{\Hom}{Hom}     
\DeclareMathOperator{\Tors}{Tors}    
\DeclareMathOperator{\Exc}{Exc}     
\DeclareMathOperator{\Pic}{Pic}     
\DeclareMathOperator{\Cl}{Cl}       
\DeclareMathOperator{\rk}{rk}       
\DeclareMathOperator{\Mov}{Mov}     
\DeclareMathOperator{\Nef}{Nef}     
\DeclareMathOperator{\Eff}{Eff}     
\DeclareMathOperator{\Relint}{Relint}  
\DeclareMathOperator{\REF}{REF}     

\title[Fibration of smooth projective toric varieties]{Fibration and classification of smooth projective toric varieties of low Picard number}

\author[M. Rossi and L.Terracini]{Michele Rossi and Lea Terracini}

\date{\today}

\address{Dipartimento di Matematica, Universit\`a di Torino,
via Carlo Alberto 10, 10123 Torino} \email{michele.rossi@unito.it,
lea.terracini@unito.it}

\thanks{The authors were partially supported by the MIUR-PRIN 2010-11 Research Funds ``Geometria delle Variet\`{a} Algebriche''. The first author is also supported by the I.N.D.A.M. as a member of the G.N.S.A.G.A.}


\def \b{\beta }

\def \s{\sigma }

\def \Ga{\Gamma }
\def \Si{\Sigma }
\def \g{\gamma}
\def \vf{\varphi}

\def \q{\mathbf{q}}
\def \pp{\mathbf{p}}

\def \u{\mathbf{u}}
\def \v{\mathbf{v}}
\def \n{\mathbf{n}}

\def \w{\mathbf{w}}
\def \t{\mathbf{t}}
\def \z{\mathbf{z}}

\def \y{\mathbf{y}}
\def \1{\mathbf{1}}
\def \0{\mathbf{0}}

\def\P{{\mathbb{P}}}

\def\p2{\mathbb{P}^2}
\def\p3{\mathbb{P}^3}
\def\p4{\mathbb{P}^4}

\def\rk{\operatorname{rk}}

\def\NE{\operatorname{NE}}
\def\Mat{\operatorname{Mat}}

\def\Z{\mathbb{Z}}

\def\C{\mathbb{C}}
\def\R{\mathbb{R}}
\def\M{\mathbf{M}}
\def\Q{\mathbb{Q}}
\def\N{\mathbb{N}}

\def\B{\mathcal{B}}

\def\CQ{\mathcal{Q}}
\def\SF{\mathcal{SF}}
\def\pc{\mathcal{P}}
\def\gkz{\mathcal{Q}}

\def\G{\mathcal{G}}

\def\Ga{\Gamma}

\def\Weil{\mathcal{W}_T}

\theoremstyle{plain}
\newtheorem{theorem}{Theorem}[section]
\newtheorem{proposition}[theorem]{Proposition}

\newtheorem{thm-def}[theorem]{Theorem--Definition}
\newtheorem{corollary}[theorem]{Corollary}

\newtheorem{lemma}[theorem]{Lemma}

\newtheorem*{a-proposition}{Proposition}

\theoremstyle{remark}

\newtheorem{example}[theorem]{Example}

\theoremstyle{definition}
\newtheorem{definition}[theorem]{Definition}

\newtheorem*{step I}{Step I}
\newtheorem*{step II}{Step II}
\newtheorem*{step III}{Step III}
\newtheorem*{step IV}{Step IV}

\newtheorem*{acknowledgements}{Acknowledgements}

\newcommand{\halfline}{\vskip6pt}
\newcommand{\longmapsfrom}{\mathrel{\reflectbox{\ensuremath{\longmapsto}}}}

\begin{document}

\begin{abstract}
In this paper we show that a smooth toric variety $X$ of Picard number $r\leq 3$ always admits a nef primitive collection supported on a hyperplane admitting non-trivial intersection with the cone $\Nef(X)$ of numerically effective divisors and cutting a facet of the pseudo-effective cone $\Eff(X)$, that is $\Nef(X)\cap\partial\overline{\Eff}(X)\neq\{0\}$. In particular this means that $X$ admits non-trivial and non-big numerically effective divisors. Geometrically this guarantees the existence of a fiber type contraction morphism over a smooth toric variety of dimension and Picard number lower  than those of $X$, so giving rise to a classification of smooth and complete toric varieties with $r\leq 3$. Moreover we revise and improve results of Oda-Miyake by exhibiting an extension of the above result to projective, toric, varieties of dimension $n=3$ and Picard number $r=4$, allowing us to classifying all these threefolds. We then improve results of Fujino-Sato, by presenting sharp (counter)examples of smooth, projective, toric varieties of any dimension $n\geq4$ and Picard number $r=4$ whose non-trivial nef divisors are big, that is $\Nef(X)\cap\partial\overline{\Eff}(X)=\{0\}$. Producing those examples represents an important goal of computational techniques in definitely setting an open geometric problem. In particular, for $n=4$, the given example turns out to be a weak Fano toric fourfold of Picard number 4.

\end{abstract}

\maketitle

\tableofcontents

\section*{Introduction}
Throughout the present paper a \emph{contraction morphism}, or simply a \emph{contraction}, is a surjective morphism with connected fibers $\vf:X\to B$ between $\Q$--factorial complete toric va\-rie\-ties.
If $\dim(B)<\dim(X)$ then $\vf$ is said \emph{of fiber type}. Otherwise $\vf$ is a \emph{birational contraction}.
In particular a contraction is called
\begin{itemize}
\item a \emph{fibration} $f:X\to B$, if $f$ exhibits $X$ as a locally trivial, torus equivariant, fiber bundle over the base $B$, with fixed fiber $F$;
\item a \emph{fibrational contraction} if it is the composition $f\circ \phi$ of a birational contraction $\phi$ and a fibration $f$; its \emph{base} and \emph{fiber} are the base and the fiber of the fibration $f$, respectively, and its \emph{exceptional locus} is the exceptional locus $\Exc(\phi)$ of the contraction $\phi$;
\item a \emph{projective toric bundle} (PTB) is the projectivization of a decomposable bundle over a smooth projective toric variety \cite[\S\,7.3]{CLS}.
\end{itemize}
As a general result, this paper is partly devoted to give a proof of the following:
\begin{theorem}\label{thm:intro}
  Let $X(\Si)$ be a smooth and complete $n$--dimensional toric variety of Picard number $2\leq r\leq 3$. Then there exists a fiber type contraction $\vf:X\twoheadrightarrow B$ over a smooth complete toric variety $B$ of dimension $m\leq n-1$ and Picard number $r_B\leq r-1$ such that either $B\cong\P^m$ or $B$ is a PTB over a projective space of smaller dimension.
  In particular we get the following:\\
  \textbf{\emph{Classification:}} if $X$ is a smooth and complete toric variety of Picard number $r\leq 3$ then  one of the following cases occurs:
  \begin{enumerate}
    \item $r=1$ and $X$ is a projective space;
    \item $r=2$ and $\vf:X\to B$ is a fibration exhibiting $X$ as a PTB over a projective space $B\cong\P^m$;
    \item $r=3$ and $\vf:X\to B$ is a fibration exhibiting $X$ as a PTB over a smooth toric variety $B$ of Picard number 2; by the previous part (2), $B$ is still a PTB over a projective space, meaning that $X$ is obtained from a projective space by a sequence of two projectivizations of decomposable toric bundles;
    \item $r=3$ and $\vf:X\to B$ is a fibrational contraction over a projective space $B\cong\P^m$, with divisorial exceptional locus and fiber given by a projective space $F\cong\P^{n-m}$;
    \item $r=3$ and $\vf:X\to B$ is a fiber type contraction over a projective space $B\cong \P^m$, which is neither a fibration nor a fibrational contraction.
  \end{enumerate}
  When $n=\dim(X)\leq 3$ the last case (5) cannot occur and $\varphi:X\to B$ is either a fibration or a fibrational contraction.
\end{theorem}
In the previous statement the word \emph{complete} is synonymous of \emph{projective}, by a celebrated result of P.~Kleinschmidt and B.~Sturmfels \cite{Kleinschmidt-Sturmfels}. The classification given in the second part of the statement descends from the first part by means, point by point,  of the following considerations.
\begin{enumerate}
\item The unique smooth and complete toric variety of Picard number (in the following also called \emph{rank}) $r=1$ is the projective space.
\item For $r=2$, the Kleinschmidt classification \cite{Kleinschmidt} exhibits every smooth complete toric variety of rank $r=2$ as a PTB over a projective space. Alternatively this point can be obtained as a byproduct of considerations given for the following part (3).
\item For $r=3$, the Kleinschmidt classification has been extended by V.~Batyrev to those smooth and complete toric varieties admitting a numerically effective (nef) \emph{primitive collection} (see the following Definition \ref{def:pc}) which is disjoint from any further primitive collection \cite[Prop.\,4.1]{Batyrev91}: this gives a PTB over a smooth toric base of Picard number $r=2$.\\
Recently, the fan's property of admitting a nef primitive collection disjoint from any further primitive collection, has been shown to be equivalent, in the projective setup, with the property of being \emph{maximally bordering} for the corresponding fan chamber $\g\cong\Nef(X)$ of the secondary fan \cite[Prop.\,3.25]{RT-Qfproj}, where \emph{bordering} means that $\g$ has non--trivial intersection with a facet of the pseudo--effective cone $\overline{\Eff}(X)$ and \emph{maximally} means that this intersection has maximal dimension (see the following Definition~\ref{def:bordering}). Since $\g\cong\Nef(X)$ (see either the following Proposition~\ref{prop:nef} or, better, \cite[Thm.\,15.1.10]{CLS}), non-trivial classes in $\Nef(X)\cap\partial\overline{\Eff}(X)$ determine nef and non-big divisors whose associated morphisms give a fiber type contraction. Namely, on the one hand, the maximally bordering case dually determines an extremal ray of $\NE(X)$ whose contraction exhibits $X$ as a PTB (see \cite[Thm\,1.5]{Reid83}, \cite[Cor.\,2.4]{Casagrande}), matching the above mentioned Batyrev's description: this fact also follows from \cite[Cor.~3.23]{RT-Qfproj} which, in addition, gives a complete description of the base and the fibre of the involved fibration.
\item[(4),(5)] On the other hand, the non-maximally bordering case exhibits an extremal ray of $\Nef(X)$, which turns out to be the pull back $\vf^*(\Nef(B))$ by a \emph{rational contraction} $\vf$, according with \cite[Prop.\,1.11.(3)]{Hu-Keel}. In particular considerations following \cite[Prop.\,2.5]{Casagrande13} prove that the associated rational contraction $\vf$ is actually regular. Moreover the bordering condition for $\g=\Nef(X)$ implies that $\vf$ is a fiber type contraction over a projective space $B\cong\P^m$, since $\dim(\vf^*(\Nef(B)))=1$: further details clarifying the difference between (4) and (5) are given in \S\,\ref{ssez:intbord}.(2).
    \end{enumerate}
    Therefore the proof of Theorem \ref{thm:intro} reduces to prove the following
    \halfline
    \noindent\textbf{Theorem \ref{thm:}} \emph{Let $X(\Si)$ be a smooth, complete, toric variety of rank $r\leq 3$ and let $\g_{\Si}=\Nef(X)$ be the associated chamber of the secondary fan. Then $\Si$ admits a nef primitive collection $\pc$ such that $\g_{\Si}$ is a bordering chamber with respect to the support hyperplane $H_P$ of $\pc$ (see Def.~\ref{def:support}), i.e. $\dim(\Nef(X)\cap H_P)\geq 1$ and $H_P$ cuts out a facet of the pseudo--effective cone $\Eff(X)$. In particular $X$ admits a non-trivial nef and non-big divisor.}
    \halfline

    Notice that the existence of a nef primitive collection $\pc$ of $\Si$ follows from \cite[Prop.~3.2]{Batyrev91}. Then the statement is obvious for $r=1$. For $2\leq r \leq 3$ the proof is obtained by means of $\Z$--linear Gale duality applied to the geometry of the secondary fan, as developed in \cite{RT-LA&GD} and \cite{RT-Qfproj}.

    For what concerns higher values of the Picard number $r$, let us point out that there is no hope of extending Theorem \ref{thm:} for $r\geq 5$: in \cite{FS} Fujino and Sato exhibited extremely ingenious examples of smooth projective toric varieties of any dimension $n\geq 3$ and rank $r\geq 5$ whose non-trivial nef divisors are big i.e. such that $\Nef(X)\cap\partial\overline{\Eff}(X)=\{0\}$. It remains then to understand what happens for:
    \begin{itemize}
      \item smooth complete and non-projective toric varieties of rank $r\geq 4$: this problem can be easily settled by observing that the famous Oda's example of a smooth complete and non-projective toric variety \cite[Prop.~9.4]{OM},\cite[p.~84]{Oda}, exhibits a 3-folds of Picard number 4 whose $\Nef$ cone is a 2-dimensional one completely internal to the pseudo-effective cone, hence showing that every non-trivial nef divisor is big; this example can be easily generalized on dimension and rank;
      \item smooth projective toric varieties of rank $r=4$: a classical Oda-Miyake result \cite[Thm.~9.6]{OM}, allows us to extend the statement of Theorem~\ref{thm:} to every projective toric \emph{threefold} (i.e. smooth 3-dimensional variety) $X$ of Picard number 4: see Theorem~\ref{thm:OM}; in particular every such threefold turns out to admitting non-big and non-trivial nef divisors; consequently we get a classification in terms of fibred type birational contractions of all toric threefolds of rank $r\leq 4$: see Theorem~\ref{thm:n=3,r=4}; based on such evidence, one might hope that Theorem \ref{thm:} could be extended to the case of projective, toric $n$-fold of rank $r=4$; unfortunately this is not the case and in \S~\ref{ssez:esempio} we give a sharp example of a projective toric 4-fold $X$ of rank $r=4$ such that $\Nef(X)\cap\partial\overline{\Eff}(X)=\{0\}$: this $X$ turns out to be already a \emph{weak Fano} toric 4-fold; such an example can be easily generalized on every dimension $n\geq 4$ (see \S~\ref{ssez:generalizzazione}), unfortunately loosing the weak Fano condition when dimension is greater than 4.
    \end{itemize}
    We believe that the latter is an interesting improvement of Fujino and Sato results, since ``in general, it seems to be hard to find those examples'' \cite[pg.~1]{FS}. In fact examples in \S~\ref{ssez:esempio} and \S~\ref{ssez:generalizzazione} have to be considered among main results of the present paper. Let us remark that we could produce them only by means of Maple procedures stressing smoothness conditions and combinatoric a\-na\-ly\-sis on the secondary fan by employing $\Z$-linear Gale duality's techniques, as developed in \cite{RT-LA&GD}. We then performed a geometric analysis of such an example following \cite{RT-Qfproj}. The construction of examples in \S~\ref{ssez:esempio} and \S~\ref{ssez:generalizzazione} shows the utility of  the employment of computational geometry routines in definitely setting open theoretic questions. In fact it would be probably impossible to find those examples by hand, following appropriate ideas and geometric constructions, as Fujino and Sato did in \cite{FS} for examples of Picard number greater than 4.

    The paper is organized as follows. After some preliminaries, in \S\,\ref{sez:pre} we state Theorem \ref{thm:}. \S\,\ref{sez:r=2} and \S\,\ref{sez:r=3} are devoted to prove this theorem when $r=2$ and $r=3$, respectively. \S\,\ref{ssez:intbord} is dedicated to giving necessary details proving Theorem~\ref{thm:intro} starting from Theorem \ref{thm:}. In \S\,\ref{ssez:ex} we discuss two examples giving all the possible situations occurring in the classification of Theorem~\ref{thm:intro}. Finally \S\,\ref{sez:r>3} is devoted to discussing the generalization of Theorem~\ref{thm:} to higher values of the Picard number.

\section{Preliminaries, notation and the general result}\label{sez:pre}
In the present paper we deal with either $\Q$--factorial or smooth projective toric varieties, hence associated with either simplicial or regular, respectively, complete fans. For preliminaries and used notation on toric varieties we refer the reader to \cite[\S~1.1]{RT-LA&GD}. We will also apply $\Z$--linear Gale duality as developed in \cite[\S~2]{RT-LA&GD}.
Every time the needed nomenclature will be recalled either directly by giving the necessary definition or by reporting the precise reference. Here is a list of main notation and relative references:

\subsection{List of notation}\label{ssez:lista}\hfill\\
Let $X(\Si)$ be a $n$--dimensional toric variety and $T\cong(\C^*)^n$ the acting torus, then
\begin{eqnarray*}
  &M,N,M_{\R},N_{\R}& \text{denote the \emph{group of characters} of $T$, its dual group and}\\
  && \text{their tensor products with $\R$, respectively;} \\
  &\Si\subseteq N_{\R}& \text{is the fan defining $X$;} \\
  &\Si(i)& \text{is the \emph{$i$--skeleton} of $\Si$, which is the collection of all the}\\
  && \text{$i$--dimensional cones in $\Si$;} \\
  &r=\rk(X)&\text{is the Picard number of $X$, also called the \emph{rank} of $X$};\\
  &F^r_{\R}& \cong\R^r,\ \text{is the $\R$--linear span of the free part of $\Cl(X(\Si))$};\\
  &F^r_+& \text{is the positive orthant of}\ F^r_{\R}\cong\R^r;\\
  &\langle\v_1,\ldots,\v_s\rangle\subseteq\N_{\R}& \text{denotes the cone generated by the vectors $\v_1,\ldots,\v_s\in N_{\R}$;}\\
  && \text{if $s=1$ then this cone is also called the \emph{ray} generated by $\v_1$;} \\
  &\mathcal{L}(\v_1,\ldots,\v_s)\subseteq N& \text{denotes the sublattice spanned by $\v_1,\ldots,\v_s\in N$;}\\
  &\mathcal{L}_{\R}(\v_1,\ldots,\v_s)& :=\mathcal{L}(\v_1,\ldots,\v_s)\otimes_{\Z}\R \,.
  \end{eqnarray*}
Let $A\in\mathbf{M}(d,m;\Z)$ be a $d\times m$ integer matrix, then
\begin{eqnarray*}
  &\mathcal{L}_r(A)\subseteq\Z^m& \text{denotes the sublattice spanned by the rows of $A$;} \\
  &\mathcal{L}_c(A)\subseteq\Z^d& \text{denotes the sublattice spanned by the columns of $A$;} \\
  &A_I\,,\,A^I& \text{$\forall\,I\subseteq\{1,\ldots,m\}$ the former is the submatrix of $A$ given by}\\
  && \text{the columns indexed by $I$ and the latter is the submatrix of}\\
  && \text{$A$ whose columns are indexed by the complementary }\\
  && \text{subset $\{1,\ldots,m\}\backslash I$;} \\
  &\REF& \text{Row Echelon Form of a matrix;}\\
  &\text{\emph{positive}}\ (\geq 0)& \text{a matrix (vector) whose entries are non-negative.}\\
\end{eqnarray*}
Given a $F$--matrix $V=(\v_1,\ldots,\v_{n+r})\in\mathbf{M}(n,n+r;\Z)$ (see Definition \ref{def:Fmatrice} below), then
\begin{eqnarray*}
  &\langle V\rangle& =\langle\v_1,\ldots,\v_{n+r}\rangle\subseteq N_{\R}\ \text{denotes the cone generated by the columns of $V$;} \\
  &\langle V\rangle(i)& =\ \text{is the set of all the $i$--dimensional faces of}\ \langle V\rangle\,;\\
  &\SF(V)& =\SF(\v_1,\ldots,\v_{n+r})\ \text{is the set of all rational simplicial fans $\Si$ such that}\\
  && \text{$\Sigma(1)=\{\langle\v_1\rangle,\ldots,\langle\v_{n+r}\rangle\}\subset N_{\R}$ and $|\Si|=\langle V\rangle$ \cite[Def.~1.3]{RT-LA&GD};}\\
  &\P\SF(V)&:=\{\Si\in\SF(V)\ |\ X(\Si)\ \text{is projective}\};\\
  &\G(V)&=Q\ \text{is a \emph{Gale dual} matrix of $V$ \cite[\S~3.1]{RT-LA&GD};} \\
  &\CQ&=\langle\G(V)\rangle \subseteq F^r_+\ \text{is a \emph{Gale dual cone} of}\ \langle V\rangle:\ \text{it is always assumed to be}\\
  && \text{generated in $F^r_{\R}$ by the columns of a positive $\REF$ matrix $Q=\G(V)$}\\
  && \text{(see \cite[Thms.~3.8,~3.18]{RT-LA&GD})}.
\end{eqnarray*}

\subsection{$\Z$--linear Gale duality}
 Let us start by recalling the $\Z$--linear interpretation of Gale duality following the lines of \cite[\S~3]{RT-LA&GD}, to which the interested reader is referred for any further detail.

\begin{definition} A $n$--dimensional $\Q$--factorial complete toric variety $X=X(\Si)$ of rank $r$ is the toric variety defined by a $n$--dimensional \emph{simplicial} and \emph{complete} fan $\Si$ such that $|\Si(1)|=n+r$ \cite[\S~1.1.2]{RT-LA&GD}. In particular the rank $r$ coincides with the Picard number i.e. $r=\rk(\Pic(X))$. A matrix $V\in\Mat(n,n+r;\Z)$, whose columns are given by elements of the monoids $\rho\cap N$, one for every $\rho\in\Si(1)$, is called a \emph{fan matrix of $X$}. If every such element is actually a generator of the associated monoid $\rho\cap N$, i.e. it is a \emph{primitive} element of the ray $\rho$, then $V$ is called a \emph{reduced} fan matrix.
\end{definition}

\begin{definition}\cite[Def.~3.10]{RT-LA&GD}\label{def:Fmatrice} An \emph{$F$--matrix} is a $n\times (n+r)$ matrix  $V$ with integer entries, satisfying the conditions:
\begin{itemize}
\item[a)] $\rk(V)=n$;
\item[b)] $V$ is \emph{$F$--complete} i.e. $\langle V\rangle=N_{\R}\cong\R^n$ \cite[Def.~3.4]{RT-LA&GD};
\item[c)] all the columns of $V$ are non zero;
\item[d)] if ${\bf  v}$ is a column of $V$, then $V$ does not contain another column of the form $\lambda  {\bf  v}$ where $\lambda>0$ is real number.
\end{itemize}
A \emph{$CF$--matrix} is a $F$-matrix satisfying the further requirement
\begin{itemize}
\item[e)] the sublattice ${\mathcal L}_c(V)\subset\Z^n$ is cotorsion free, which is ${\mathcal L}_c(V)=\Z^n$ or, equivalently, ${\mathcal L}_r(V)\subset\Z^{n+r}$ is cotorsion free.
\end{itemize}
A $F$--matrix $V$ is called \emph{reduced} if every column of $V$ is composed by coprime entries \cite[Def.~3.13]{RT-LA&GD}. A (reduced) fan matrix of a $\Q$--factorial complete toric variety $X(\Si)$ is always a (reduced) $F$--matrix.
\end{definition}

\begin{definition}\cite[Def.~3.9]{RT-LA&GD}\label{def:Wmatrix} A \emph{$W$--matrix} is an $r\times (n+r)$ matrix $Q$  with integer entries, satisfying the following conditions:
\begin{itemize}
\item[a)] $\rk(Q)=r$;
\item[b)] ${\mathcal L}_r(Q)$ has not cotorsion in $\Z^{n+r}$;
\item[c)] $Q$ is \emph{$W$--positive}, which is $\mathcal{L}_r(Q)$ admits a basis consisting of positive vectors (see list \ref{ssez:lista} and \cite[Def.~3.4]{RT-LA&GD}).
\item[d)] Every column of $Q$ is non-zero.
\item[e)] ${\mathcal L}_r(Q)$   does not contain vectors of the form $(0,\ldots,0,1,0,\ldots,0)$.
\item[f)]  ${\mathcal L}_r(Q)$ does not contain vectors of the form $(0,a,0,\ldots,0,b,0,\ldots,0)$, with $ab<0$.
\end{itemize}
A $W$--matrix is called \emph{reduced} if $V=\G(Q)$ is a reduced $F$--matrix \cite[Def.~3.14, Thm.~3.15]{RT-LA&GD}.

\noindent The Gale dual matrix $Q=\G(V)$ of a fan matrix $V$ of a $\Q$--factorial complete toric variety $X(\Si)$ is a $W$--matrix called  a \emph{weight matrix of $X$}.
\end{definition}

\begin{definition}\cite[Def.~2.7]{RT-LA&GD}\label{def:PWS} A \emph{poli--weighted space} (PWS) is a $n$--dimensional $\Q$--factorial complete toric variety $Y(\widehat{\Si})$ of rank $r$, whose reduced fan matrix $\widehat{V}$ is a $CF$--matrix i.e. if
\begin{itemize}
  \item $\widehat{V}$ is a $n\times(n+r)$ $CF$--matrix,
  \item $\widehat{\Si}\in\SF(\widehat{V})$.
\end{itemize}
Let us recall that a $\Q$--factorial complete toric variety $Y$ is a PWS if and only if it is \emph{1-connected in codimension 1} \cite[\S~3]{Buczynska}, which is $\pi^1_1(Y)\cong\Tors(\Cl(Y))=0$ \cite[Thm.~2.1]{RT-QUOT}. In particular a smooth complete toric variety is always a PWS.
\end{definition}

\subsection{The secondary fan}\label{sez:secondary}
Let us now recall the $\Z$--linear interpretation of the secondary fan, as done in  \cite[\S~1.2]{RT-Qfproj}, to which the interested reader is referred for any further detail.

\begin{definition}\label{def:GKZ&Mov}\cite[Def.~1.6]{RT-Qfproj} Let $Q$ be a positive, REF and reduced $W$-matrix and $\CQ=\langle Q\rangle$. Let $\mathcal{S}_r$ be the family of all the $r$--dimensional subcones of $\CQ$ obtained as intersection of simplicial subcones of $\CQ$ generated by columns of $Q$. Then define the \emph{secondary fan} (or \emph{GKZ decomposition}) of $Q$ to be the set $\Ga=\Ga(Q)$ of cones in $F^r_+$ such that
\begin{itemize}
  \item its subset of $r$--dimensional cones (the \emph{$r$--skeleton}) $\Ga(r)$ is composed by the minimal elements, with respect to the inclusion, of the family $\mathcal{S}_r$,
  \item its subset of $i$--dimensional cones (the \emph{$i$--skeleton}) $\Ga(i)$ is composed by all the $i$--dimensional faces of cones in $\Ga(r)$, for every $1\leq i\leq r-1$.
\end{itemize}
 A maximal cone $\g\in\Ga(r)$ is called a \emph{chamber} of the secondary fan $\Ga$. Finally define
\begin{equation}\label{Mov}
    \Mov(Q):= \bigcap_{i=1}^{n+r}\left\langle Q^{\{i\}}\right\rangle\ ,
\end{equation}
where $\left\langle Q^{\{i\}}\right\rangle$ is the cone generated in $F^r_+$ by the columns of the submatrix $Q^{\{i\}}$ of $Q$ (see the list of notation \ref{ssez:lista}).
\end{definition}

The fact that $\Ga$ is actually a fan is a consequence of the following

\begin{theorem}\label{thm:GKZ}
If $V$ is a reduced $F$--matrix and $Q=\G(V)$ is a positive and REF $W$-matrix  then, for every $\Si\in\SF(V)$,
\begin{enumerate}
  \item $\CQ=\overline{\Eff}(X(\Si))$, the \emph{pseudo--effective cone of $X$}, which is the closure of the cone generated by effective Cartier divisor classes of $X$ \cite[Lemma~15.1.8]{CLS},
  \item $\Mov(Q)=\overline{\Mov}(X(\Si))$, the closure of the cone generated by movable Cartier divisor classes of $X$ \cite[(15.1.5), (15.1.7), Thm.~15.1.10, Prop.~15.2.4]{CLS}.
  \item $\Ga(Q)$ is the secondary fan (or GKZ decomposition) of $X(\Si)$ \cite[\S~15.2]{CLS}. In particular $\Ga$ is a fan and $|\Ga|=\CQ\subset F^r_+$.
\end{enumerate}
\end{theorem}

\begin{theorem}\emph{\cite[Prop. 15.2.1]{CLS}}\label{thm:Gale sui coni} There is a one to one correspondence between the following sets
\begin{eqnarray*}
    \mathcal{A}_{\Ga}(Q)&:=&\{\g\in\Ga(r)\ |\ \g\subset\Mov(Q)\}\\
    \P\SF(V)&:=&\{\Si\in\SF(V)\ |\ X(\Si)\ \text{is projective}\}\ .
\end{eqnarray*}
\end{theorem}

For the following it is useful to understand the construction of such a correspondence. Namely (compare \cite[Prop.~15.2.1]{CLS}):
\begin{itemize}
  \item after \cite{Berchtold-Hausen}, given a chamber $\g\in\mathcal{A}_{\Ga}$ let us call \emph{the bunch of cones of $\g$} the collection of cones in $F^r_+$ given by
      \begin{equation*}
        \mathcal{B}(\g):=\left\{\left\langle Q_J\right\rangle\ |\ J\subset\{1,\ldots,n+r\}, |J|=r, \det\left(Q_J\right)\neq 0, \g\subset\left\langle Q_J\right\rangle\right\}
      \end{equation*}
      (see also \cite[p.~738]{CLS}),
  \item it turns out that
  \begin{equation}\label{camera}
    \g=\bigcap_{\b\in\mathcal{B}(\g)}\b\,,
  \end{equation}
  \item  for any $\g\in\mathcal{A}_{\Ga}(Q)$ there exists a unique fan $\Si_{\g}\in\P\SF(V)$ such that
  \begin{equation*}
    \Si_{\g}(n):=\left\{\left\langle V^J\right\rangle\ |\ \left\langle Q_J\right\rangle\in \mathcal{B}(\g)\right\}\,,
  \end{equation*}
  \item for any $\Si\in\P\SF(V)$ the collection of cones
  \begin{equation}\label{bunch}
    \mathcal{B}_{\Si}:=\left\{\left\langle Q^I\right\rangle\ |\ \left\langle V_I\right\rangle\in\Si(n)\right\}
  \end{equation}
  is the bunch of cones of the chamber $\g_{\Si}\in\mathcal{A}_{\Ga}$ given by $\g_{\Si}:=\bigcap_{\b\in\mathcal{B}_{\Si}}\b$.
\end{itemize}
Then the correspondence in Theorem \ref{thm:Gale sui coni} is realized by setting
\begin{equation*}
\begin{array}{ccc}
  \mathcal{A}_{\Ga}(Q) & \longleftrightarrow & \P\SF(V) \\
  \g & \longmapsto & \Si_{\g} \\
  \g_{\Si} & \longmapsfrom & \Si
\end{array}
\end{equation*}
Let us moreover recall that if $V$ is a $CF$--matrix then
\begin{equation}\label{Gale su det}
  \forall\,I\subset\{1,\ldots,n+r\}\ :\ |I|=n \quad |\det(V_I)|=|\det(Q^I)|
\end{equation}
\cite[Cor.\,3.3]{RT-LA&GD}. Therefore a regular fan determines a \emph{regular} bunch of cones and viceversa.

As a final result let us recall the following
\begin{proposition}\emph{\cite[Thm. 15.1.10(c)]{CLS}}\label{prop:nef} If $V=(\v_1\ldots,\v_{n+r})$ is a $F$--matrix then, for every fan $\Si\in\P\SF(V)$, there is a na\-tu\-ral isomorphism $\Pic(X(\Si))\otimes\R\cong F^r_{\R}$ taking the cones
\begin{equation*}
    \Nef(X(\Si))\subseteq\overline{\Mov}(X(\Si))\subseteq\overline{\Eff}(X(\Si))
\end{equation*}
to the cones
\begin{equation*}
    \g_{\Si}\subseteq\Mov(Q)\subseteq\CQ\,.
\end{equation*}
In particular, calling $d:\mathcal{W}_T(X(\Si))\to\Cl(X(\Si))$  the morphism giving to a torus invariant divisor $D$ its linear equivalence class $d(D)$, we get that a Weil divisor $D$ on $X(\Si)$ admits a nef (ample) positive multiple if and only if $d(D)\in\g_{\Si}$ (\,$d(D)\in\Relint\left(\g_{\Si}\right)$, where $\Relint$ denotes the interior of the cone $\g_{\Si}$ in its linear span).
 \end{proposition}

 \subsection{Primitive collections and relations}\label{sez:primitive}
 Let $V=(\v_1,\ldots,\v_{n+r})$ be a reduced $F$--matrix and consider a fan $\Si\in\SF(V)$. The datum of a collection of rays $\pc=\{\rho_1,\ldots,\rho_k\}\subseteq\Si(1)$ determines a subset $P=\{j_1,\ldots,j_k\}\subseteq\{1,\ldots,n+r\}$ such that
\begin{equation*}
    \pc=\{\rho_1,\ldots,\rho_k\}=\left\{\langle\v_{j_1}\rangle,\ldots,\langle\v_{j_k}\rangle\right\}
\end{equation*}
and a submatrix $V_P$ of $V$.
By abuse of notation we will often write
\begin{equation*}
    \pc=\{\v_{j_1},\ldots,\v_{j_k}\}=\{V_P\}\,.
\end{equation*}
From the point of view of the Gale dual cone $\CQ=\langle Q\rangle$, where $Q=(\q_1,\ldots,\q_{n+r})=\G(V)$ is a reduced, positive, REF, $W$--matrix, the subset $P\subseteq\{1,\ldots,n+r\}$ determines the collection $\pc^*=\{\langle\q_{j_1}\rangle,\ldots,\langle\q_{j_k}\rangle\}\subseteq\Ga(1)$. By the same abuse of notation we will often write
\begin{equation*}
    \pc^*=\{\q_{j_1},\ldots,\q_{j_k}\}=\{Q_P\}\,.
\end{equation*}
The vector $\v_{\mathcal{P}}:=\sum_{i=1}^k \v_{j_i}$ lies in the relative interior of a cone $\s\in\Si$ and there is a unique relation
\begin{equation}\label{relazione primitiva 2}
    \v_{\mathcal{P}}-\sum_{\rho\in\s(1)}c_{\rho}\v_{\rho}=0\ ,\quad\text{with $c_{\rho}\in\Q\,,\,c_{\rho}> 0$}
\end{equation}
where $\v_{\rho}$ is the column of $V$ generating $\rho\cap N$ as a monoid.
This fact allows us to define a rational vector $r(P)=r(\mathcal{P})=(b_1,\ldots,b_{n+r})\in\Q^{n+r}$, where $b_j$ is the coefficient of the column $\v_j$ of $V$ in (\ref{relazione primitiva 2}). Let $l$ be the least common denominator of $b_1,\ldots,b_{n+r}$. Then,
\begin{equation}\label{r_Z(P)}
    r_{\Z}(P)=r_{\Z}(\mathcal{P}):=l\,r(\mathcal{P})=(lb_1,\ldots,lb_{n+r})\in\mathcal{L}_r(Q)\subset\Z^{n+r}\,.
\end{equation}

\begin{definition}\label{def:pc}
  A collection $\mathcal{P}:=\left\{\rho_1,\ldots\rho_k\right\}\subset\Si(1)$ is called a \emph{primitive collection for $\Si$} if $\mathcal{P}$ is not contained in a single cone of $\Si$ but every proper subset is (compare \cite[Def.~2.6]{Batyrev91}, \cite[Def.~1.1]{Cox-vRenesse}, \cite[Def.~5.1.5]{CLS}).
\end{definition}

 If $\mathcal{P}$ is a primitive collection then it is determined by the positive entries of $r_{\Z}(\mathcal{P})$ in (\ref{r_Z(P)}). For the details see \cite[Lemma~1.8]{Cox-vRenesse}). This is no more the case if $\mathcal{P}$ is not a primitive collection.

  \noindent Recall that the transposed fan matrix $V^T$ and the transposed weight matrix $Q^T$ of a Gale dual $Q=\G(V)$ are representative matrices of the morphisms $div$ and $d^{\vee}$, respectively, in the following standard exact sequence
 \begin{equation}\label{div-sequence}
    \xymatrix{0 \ar[r] & M\ar[r]^-{div}_-{V^T}& \mathcal{W}_T(X)=\Z^{n+r}
\ar[r]^-{d} & \Cl(X) \ar[r]& 0}
\end{equation}
and in the dual sequence of it
\begin{equation}\label{HomZ-div-sequence}
                    \xymatrix{0 \ar[r] & A_1(X):=\Hom(\Cl(X),\Z) \ar[r]^-{d^{\vee}}_-{Q^T} & \Hom(\mathcal{W}_T(X),\Z)= \Z^{n+r}
 \ar[r]^-{div^{\vee}}_-{V} & N }
\end{equation}
Then (\ref{r_Z(P)}) gives that $r_{\Z}(P)\in\im(d^{\vee})$. Since $d^{\vee}$ is injective there exists a unique $\n_P\in A_1(X)$ such that
\begin{equation}\label{normalP}
    d^{\vee}(\n_P)=Q^T\cdot\n_P=r_{\Z}(P)
\end{equation}
which turns out to be the numerical equivalence class of the 1-cycle $r_{\Z}(P)$, whose intersection index with the torus--invariant Weil divisor $lD_j$, where $D_j$ denotes the closure of the torus orbit of the ray $\langle\v_j\rangle$, is given by the integer $lb_j$, for every $1\leq j\leq n+r$. In particular, given a primitive collection $\pc$ the associated primitive relation $r_{\Z}(P)$ is a numerically effective 1-cycle (nef) if and only if all the coefficients $lb_j$ in (\ref{r_Z(P)}) are non-negative: in this case $\pc$ will be called \emph{a numerically effective (nef) primitive collection}.

\begin{definition}\label{def:support} Given a collection $\pc=\{V_P\}$, for $P\subseteq\{1,\ldots,n+r\}$, its associated numerical class $\n_{P}\in N_1(X):=A_1(X)\otimes\R$, defined in (\ref{normalP}), determines a unique dual hyperplane
$$H_P\subseteq F^r_{\R}=\Cl(X)\otimes\R$$
which is called \emph{the support} of $\pc$.
\end{definition}

The following proposition gives some further characterization of a primitive collection.

\begin{proposition}\emph{\cite[Prop.\,2.2]{RT-Qfproj}}\label{prop:primitive} Let $V$ be a reduced $F$--matrix, $Q=\G(V)$ be a Gale dual $\REF$, positive $W$--matrix, $\Si\in\P\SF(V)$ and $P\subseteq\{1,\ldots,n+r\}$ such that $\pc=\{V_P\}$ is a primitive collection for $\Si$. Then $2\leq |\pc|=|P|\leq n+1$ and the following facts are equivalent:
\begin{enumerate}
  \item $\pc$ is a primitive collection for $\Si$, which is
  \begin{itemize}
    \item[($i.1$)] $\forall\,\s\in\Si(n)\quad\pc\nsubseteq\s(1)$,
    \item[($ii.1$)] $\forall\,\rho_i\in\pc\quad\exists\,\s\in\Si(n):\pc\backslash\{\rho_i\}\subseteq\s(1)$;
  \end{itemize}
  \item $V_P$ is a submatrix of $V$ such that
  \begin{itemize}
    \item[($i.2$)] $\forall\,J\subseteq\{1,\ldots,n+r\}:\langle V_J\rangle\in\Si(n)\quad\langle V_P\rangle\nsubseteq\langle V_J\rangle$,
    \item[($ii.2$)] $\forall\,i\in P\quad\exists\,J\subseteq\{1,\ldots,n+r\}:\langle V_J\rangle\in\Si(n)\ ,\ \langle V_{P\backslash\{i\}}\rangle\subseteq\langle V_J\rangle$;
  \end{itemize}
  \item $Q^P$ is a submatrix of $Q=\G(V)$ such that
  \begin{itemize}
    \item[($i.3$)] $\forall\,J\subseteq\{1,\ldots,n+r\}:\langle Q^J\rangle\in\mathcal{B}(\g_{\Si})\quad\langle Q^J\rangle\nsubseteq\langle Q^P\rangle$,
    \item[($ii.3$)] $\forall\,i\in P\quad\exists\,J\subseteq\{1,\ldots,n+r\}:\langle Q^J\rangle\in\mathcal{B}(\g_{\Si})\ ,\ \langle Q^J\rangle\subseteq\langle Q^{P\backslash\{i\}}\rangle$;
  \end{itemize}
  \item $Q^P$ is a submatrix of $Q=\G(V)$ such that
  \begin{itemize}
    \item[($i.4$)] $\g_{\Si}\nsubseteq\langle Q^P\rangle$,
    \item[($ii.4$)] $\forall\,i\in P\quad\g_{\Si}\subseteq\langle Q^{P\backslash\{i\}}\rangle$.
   \end{itemize}
\end{enumerate}
Moreover the previous conditions ($ii.1$), ($ii.2$), ($ii.3$), ($ii.4$) are equivalent to the following one:
\begin{itemize}
  \item[($ii$)] $\forall\,i\in P\quad\exists\, \mathcal{C}_{i,P}\in\mathcal{B}(\g_{\Si}):\ \mathcal{C}_{i,P}(1)\cap\pc^*=\{\langle\q_{i}\rangle\}$\,.
\end{itemize}
\end{proposition}

\subsection{The general result: when a smooth chamber is bordering?}
Let us first of all introduce the following

\begin{definition}[\emph{Bordering} collections and chambers]\cite[Def.\,3.5]{RT-Qfproj}\label{def:bordering}\hfill
\begin{enumerate}
  \item Let $V$ be a reduced $F$--matrix and $Q=\G(V)$ a $\REF$, positive $W$--matrix. A collection $\pc=\{V_P\}$, for some $P\in\mathfrak{P}$, is called \emph{bordering} if its support $H_P$ cuts out a facet of the Gale dual cone $\gkz=\langle Q\rangle$.
  \item A chamber $\g\in\Ga(Q)$ is called \emph{bordering} if $\dim(\g\cap\partial\gkz)\geq 1$. Notice that $\g\cap\partial\gkz$ is always composed by faces of $\g$: if it contains a facet of $\g$ then $\g$ is called \emph{maximally bordering} (\emph{maxbord} for short).  A hyperplane $H$ cutting a facet of $\gkz$ and such that $\dim(\g\cap H)\geq 1$ is called a \emph{bordering hyperplane of $\g$} and the bordering chamber $\g$ is also called \emph{bordering with respect to $H$}. A normal vector $\n$ to a bordering hyperplane $H$ is called \emph{inward} if $\n\cdot x\geq 0$ for every $x\in\g$.
  \item A bordering chamber $\g\in\Ga(Q)$, w.r.t. the hyperplane $H$,
is called \emph{internal bordering} (\emph{intbord} for short) w.r.t. $H$, if either $\g$ is maxbord w.r.t $H$ or there exists an hyperplane $H'$, cutting a facet of $\g$ and such that
\begin{itemize}
  \item[($i$)] $\g\cap H\subseteq \g\cap H'$
  \item[($ii$)] $\exists\,\q_1,\q_2\in H\cap\gkz(1) : (\n'\cdot\q_1)(\n'\cdot\q_2)< 0$
 \end{itemize}
where $\n'$ is the inward primitive normal vector of $H'$.
\end{enumerate}
Notice that a primitive collection $\pc=\{V_P\}$ is bordering if and only $\pc$ is nef, i.e. $r_{\Z}(\pc)$ is a numerically effective 1-cycle.
\end{definition}

\begin{theorem}\label{thm:} Let $X(\Si)$ be a smooth complete toric variety of rank $r\leq 3$ (hence projective, by \cite{Kleinschmidt-Sturmfels})  and $\g_{\Si}\subseteq\Mov(Q)$ be the associated chamber of the secondary fan. Then $\Si$ admits a nef primitive collection $\pc$ such that $\g_{\Si}$ is a bordering chamber with respect to the support hyperplane $H_P$ of $\pc$. In particular $X$ admits a non-trivial nef and non-big divisor.
\end{theorem}

\section{The case of Picard number $r=2$}\label{sez:r=2}
Let us start by proving Theorem \ref{thm:} for $r=2$.
\begin{figure}
\begin{center}
\includegraphics[width=8truecm]{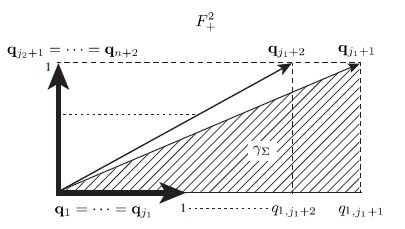}
\caption{\label{fig1}The Gale dual cone $\mathcal{Q}=F^2_+$.}
\end{center}
\end{figure}

Let $V=(\v_1,\ldots,\v_{n+2})\in\Mat(n,n+2;\Z)$ be a reduced fan matrix of $X(\Si)$ and $Q=\G(V)=(\q_1,\ldots,\q_{n+2})\in\Mat(2,n+2;\Z)$ be an associated positive $\REF$ weight matrix. Recalling the argument proving \cite[Thm.\,2.1]{RT-Qfproj}, and revising the Kleinschmidt argument for classifying smooth and complete toric varieties of rank $r=2$  \cite{Kleinschmidt}, up to left multiplication by a unimodular matrix and a possible rearrangement of columns, the $W$--matrix $Q$ and the chamber $\g_{\Si}$ can be assumed to be in the following form
\begin{equation}\label{Q1}
    Q=\overbrace{\left(
        \begin{array}{ccc}
          1 & \cdots & 1\\
          0 & \cdots & 0 \\
          \end{array}\right.}^{j_1\geq 2}
          \left.
        \begin{array}{cccccc}
           q_{1,j_1+1}&\cdots & q_{1,j_2} & \mathbf{0}_{n+2-j_2}  \\
           1 & \cdots & 1 & \mathbf{1}_{n+2-j_2} \\
        \end{array}
      \right)\quad\text{with}\quad\g_{\Si}=\left\langle\q_{j_1},\q_{j_1+1}\right\rangle
\end{equation}
as described in Fig.\,\ref{fig1}.
Then $\pc:=\{\v_j\,|\,j\geq j_1+1\}$ is a primitive nef collection of $\Si$, whose primitive relation is given by the bottom row of $Q$. Moreover $\g_{\Si}$ is clearly maxbord w.r.t. the support hyperplane $H_P=\{x_2=0\}\subset\Cl(X)\otimes\R\cong \R^2$.

\section{The case of Picard number $r=3$}\label{sez:r=3}

The proof of Theorem~\ref{thm:} for $r=3$ turns out to be significantly more intricate than the previous case $r=2$.

Let us start by considering a nef primitive collection $\pc=\{V_P\}$ for a \emph{projective} fan $\Si\in\P\SF(V)$, whose support is a plane $H_P\subset F^3_{\R}$ cutting out a facet of $\gkz\subseteq F^3_+$: such a primitive collection  certainly exists when $\Si$ is regular, by \cite[Prop.\,3.2]{Batyrev91}.

\begin{figure}
\begin{center}
\includegraphics[width=10truecm]{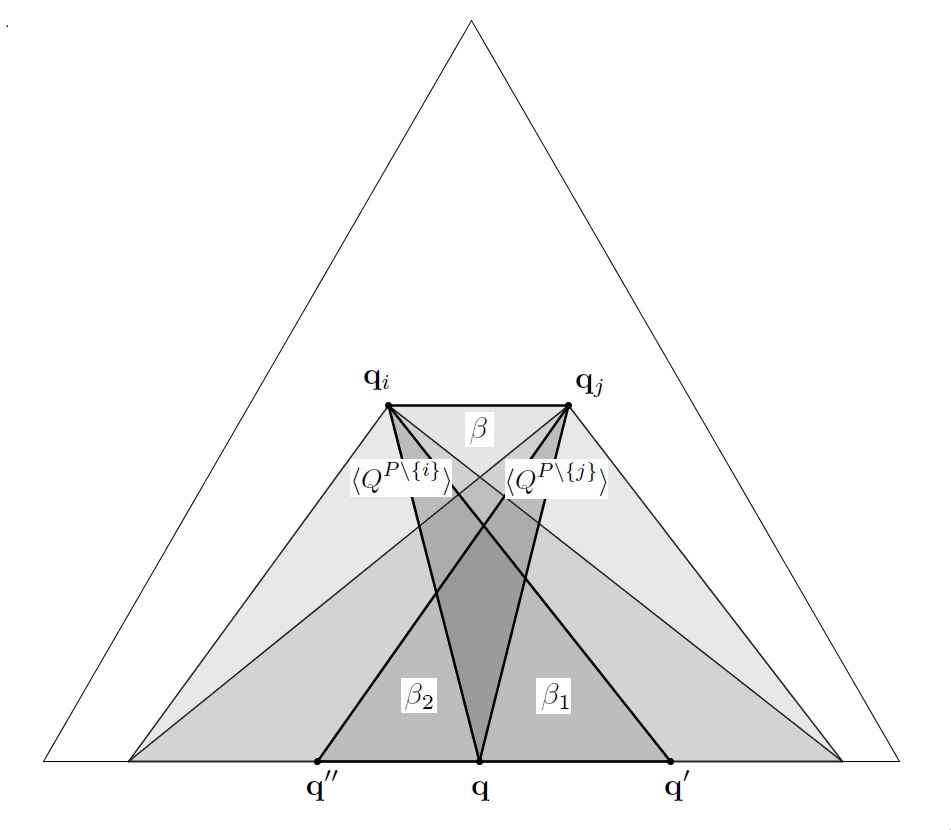}
\caption{\label{fig2}The section of the positive orthant $F^3_+$, cut out by the plane $\sum_{i=1}^3x_i=1$, describing the situation of Lemma \ref{lm:intersezione}.}
\end{center}
\end{figure}

\begin{lemma}\label{lm:bordante}
  $\forall\,\b\in\B(\g_{\Si})\quad \b(1)\cap\{Q^P\}\neq\emptyset$
\end{lemma}
Since this result holds for every $r\geq 1$, we postpone its proof to the section \S\ref{sez:r>3}

Let us now set $\B'(\g_{\Si}):=\{\b\in \B(\g_{\Si})\,:\,|\b(1)\cap\{Q^P\}|=2\}$. Then we get the following

\begin{lemma}\label{lm:intersezione}
 $\forall\,\b\in \B(\g_{\Si})\setminus\B'(\g_{\Si})\quad\exists\b_1,\b_2\in\B'(\g_{\Si})\,:\,\b_1\cap\b_2\subseteq\b$\,.
\end{lemma}
\begin{proof} First of all recall that every $\b\in\B(\g_\Si)$ is a simplicial cone generated by $r=3$ vectors in $F^3_+$, by the definition of a bunch of cones we gave above in \S~\ref{sez:secondary}. Then
  Lemma \ref{lm:bordante} implies that $|\b(1)\cap\pc^*|=2$ for every $\b\in\B(\g_{\Si})\setminus\B'(\g_{\Si})$, since $\pc^*$ is given by all columns of $Q$ not belonging to $\{Q^P\}$, as defined in \S~\ref{sez:primitive}. Therefore $$\b=\langle\q_i,\q_j,\q\rangle\,,\quad\text{with $\q_i,\q_j\in\pc^*$ and $\q\in H_P$.}$$
  $\pc$ is a primitive collection, meaning that $\g_{\Si}\subseteq\langle Q^{P\setminus\{i\}}\rangle\cap\langle Q^{P\setminus\{j\}}\rangle$, by Proposition~\ref{prop:primitive}.($ii.4$). Notice that $\langle Q^{P\setminus\{i\}}\rangle$ is a cone admitting the unique ray $\langle\q_i\rangle$ outside the support plane $H_P$, and analogously for $\langle Q^{P\setminus\{j\}}\rangle$: see Fig.~\ref{fig2} for fixing ideas.

  In particular $\g_{\Si}\subseteq \langle Q^{P\setminus\{i\}}\rangle\cap\b$, meaning that there exists a column $\q'$ of $Q$ in $ H_P$, lying on the same side of $\q_j$ with respect to the facet $\langle \q_i,\q\rangle$ of $\beta$, and such that $\g_{\Si}\subseteq\langle\q_i,\q',\q\rangle$; otherwise $\g_{\Si}$ would be included in the 2-dimensional cone $\langle\q_i,\q\rangle$, which is not possible as the fan $\Si$ determines a projective variety.  Analogously there exists a column $\q''$ of $Q$ in $ H_P$ lying on the same side of $\q_i$ with respect to the facet $\langle \q_j,\q\rangle$ of $\beta$, and such that $\g_{\Si}\subseteq\langle\q_j,\q',\q\rangle$. 	Set $\b_1:=\langle\q_i,\q',\q\rangle\,,\,\b_2:=\langle\q_j,\q'',\q\rangle$; then it is easy to verify that $\beta_1\cap\beta_2\subseteq\beta$. Moreover $\beta_1,\beta_2\in\B(\g_{\Si})\setminus\B'(\g_{\Si})$, as $\q,\q',\q''\in Q^P$.

\end{proof}

\begin{corollary}\label{cor:intersezione}
$\g_{\Si}=\bigcap_{\b\in\B'(\g_{\Si})} \b$
\end{corollary}

\begin{proof}
  By Lemma \ref{lm:intersezione}, clearly $\g_{\Si}\subseteq\bigcap_{\b\in\B'(\g_{\Si})} \b$.\\
  For the converse, notice that $x\not\in\g_{\Si}$ implies $x\not\in\bigcap_{\b\in\B'(\g_{\Si})} \b$. In fact, recalling relation (\ref{camera}) following Theorem~\ref{thm:Gale sui coni}, if $x\not\in\g_{\Si}$ then there exists $\b\in\B(\g_{\Si})$ such that $x\not\in\b$. Then Lemma~\ref{lm:intersezione} gives two cones $\b_1,\b_2\in\B'(\g_{\Si})$ such that $x\not\in\b_1\cap\b_2$.
  \end{proof}
\begin{figure}
\begin{center}
\includegraphics[width=8truecm]{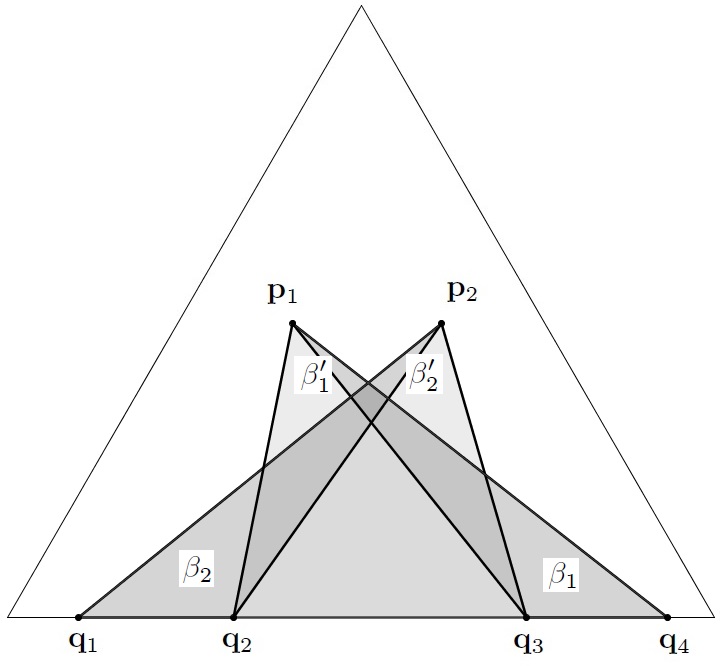}
\caption{\label{fig3}A smooth chamber cannot be no bordering}
\end{center}
\end{figure}
Consequently, Corollary \ref{cor:intersezione} allows us to conclude the proof of Theorem~\ref{thm:} by observing that the situation described in Fig.\,\ref{fig3} cannot occur when $\Si$ is a regular fan.
In fact, consider $\b_1=\langle\pp_1,\q_3,\q_4\rangle, \b_2=\langle\pp_2,\q_1,\q_2\rangle\in\B'(\g_{\Si})$ with $\pp_i\in\pc^*$ and $\q_j\in H_P$, for $1\leq i\leq 2$ and $1\leq j\leq 4$.
Since $H_P$ is a plane, $\q_3\not\in\langle\q_1,\q_2\rangle$ and $\q_1,\q_2$ are linearly independent, then
\begin{equation*}
  \exists\,a,b\in\Z\ :\ \q_3=a\q_1+b\q_2\quad\text{and}\quad ab<0\,.
\end{equation*}
Let $\n$ be the inward normal vector to the facet of $\b_2$ passing through $\pp_2,\q_1$. Clearly the cone $\b_2'=\langle\pp_2,\q_1,\q_3\rangle$ belongs to $\B'(\g_{\Si})$. Then, recalling condition (\ref{Gale su det}), the regularity of $\Si$ imposes that
\begin{equation*}
  1=\det(\pp_2,\q_1,\q_3)=\n\cdot( a\q_1+b\q_2)=b \ \Rightarrow\ a<0\,.
\end{equation*}
Analogously, $\b_1'=\langle\pp_1,\q_2,\q_4\rangle\in\B'(\g_{\Si})$ gives $\q_2=a'\q_4+\q_3$ with $a'<0$. Then
\begin{equation*}
  \q_3=a\q_1+\q_2=a\q_1+a'\q_4+\q_3\ \Rightarrow\ a\q_2+a'\q_4=0
\end{equation*}
contradicting the positivity of both $\q_2$ and $\q_4$.

\subsection{An alternative proof of Theorem \ref{thm:}}\label{ssez:batyrev}

We are now going to sketch an alternative proof of Theorem \ref{thm:} obtained by working out a comment of Cinzia Casagrande to a first draft of the present paper. In spite of the fact that the following proof may look easier than the previous one, notice that the former requires the deep Batyrev's analysis of primitive relations, given in \cite[\S~5,~\S~6]{Batyrev91} and needed to prove its Thm.\,6.6. Moreover the previous proof shed some light on the case of higher rank values, involving results whose holdness goes beyond the bound $r\leq 3$ (as for Lemma \ref{lm:bordante} and Theorem \ref{thm:OM}), on the contrary of Batyrev's analysis of primitive relations which is strictly related with the Picard number $r=3$.

First of all, let us recall that \cite[Thm.\,1.6]{Cox-vRenesse} exhibits the Mori cone $\NE(X(\Si))$ as the cone generated by the primitive relations of $\Si$. Dually this means that every hyperplane supporting a primitive collection either cuts out a facet of $\Nef(X(\Si))=\g_{\Si}$ or is a positive linear combination of some further supporting hyperplanes. When $r=2$ bordering means maximally bordering and the Kleinschmidt classification and Batyrev's considerations given in \cite[\S\,4]{Batyrev91} end up the proof. Let us then assume $r=3$. Then \cite[Thm.\,5.7]{Batyrev91} shows that $\Si$ admits either 3 or 5 primitive relations. In the first case, $\g_{\Si}$ turns out to be simplicial and \cite[Prop.\,4.1,\,Thm.\,4.3]{Batyrev91} show that $\Si$ is a splitting fan admitting a nef primitive collection, hence proving that $\g_{\Si}$ is (maximally) bordering. In the second case, \cite[Thm.~5.7,~Thm.~6.6]{Batyrev91} shows that $\Si$ has 5 primitive relations given by
\begin{eqnarray}
  \v + \y &=& c\z + b\t + \t \label{1}\\
  \y + \z &=& \u \label{2}\\
  \z + \t &=& 0 \label{3}\\
  \t + \u &=& \y \label{4}\\
  \u + \v &=& c\z + b\t \label{5}
\end{eqnarray}
where $$\v:=\sum_{i=1}^v \v_i\,,\,\y:=\sum_{i=1}^y \y_i\,,\,c\z:=\sum_{i=1}^z c_i\z_i\,,\,\v:=\sum_{i=1}^v \v_i\,,\,b\t:=\sum_{i=1}^t b_i\t_i\,,\,\u:=\sum_{i=1}^u \u_i$$
 with $\{\v_1,\ldots,\v_v\}$, $\{\y_1,\ldots,\y_y\}$, $\{\z_1,\ldots,\z_z\}$, $\{\t_1,\ldots,\t_t\}$, $\{\u_1,\ldots,\u_u\}$ di\-sjoint subsets of columns of a fan matrix $V$ of $X$ such that $v+y+z+t+u=n+3$ and $b_1,\ldots,b_t,c_1,\ldots,c_z$ non\-ne\-ga\-tive integral coefficients. Therefore $(\ref{3})=(\ref{2})+(\ref{4})$ and $(\ref{5})=(\ref{1})+(\ref{4})$, showing that $\g_{\Si}$ is simplicial and bordering w.r.t. the support of the nef primitive relation (\ref{3}).

\subsection{About the proof of Theorem \ref{thm:intro}}\label{ssez:intbord}

Given a nef primitive collection $\pc=\{V_P\}$ whose support $H_P$ cuts out a facet of $\gkz=\langle Q\rangle$, Theorem \ref{thm:} ensures that the chamber $\g_{\Si}$ of a regular fan $\Si$ is bordering w.r.t. $H_P$. When $r=3$, there are three cases:
\begin{enumerate}
  \item $\g_{\Si}$ is maxbord w.r.t $H_P$, which is $\dim(\g_{\Si}\cap H_P)=2$: in this case there exists a fibration $f:X(\Si)\to B$ over a smooth toric variety $B$ of rank $s=2$ and dimension $m<n$, hence $B$ is a PTB over a projective space \cite{Kleinschmidt}, whose fibre $F$ has rank $r-s=1$, giving $F\cong\P^{n-m}$ and proving part (3) of Theorem~\ref{thm:intro}. As already mentioned in the introduction, the existence of $f$ can be deduced as a consequence of one of the following considerations:
      \begin{itemize}
        \item as the morphism associated with the linear system of any non trivial divisor whose class is in $\g_{\Si}\cap H_P$, which is clearly nef and non-big,
        \item $H_P$ cuts out a facet of $\g_{\Si}=\Nef(X(\Si))$, hence dually determines the extremal ray of the Mori cone $\NE(X)$ generated by the inward normal vector $\n_P$, which is the numerical class of the primitive relation $r_{\Z}(\pc)$: think of $f$ as the contraction of $\langle\n_P\rangle$, in the sense of Mori Theory, which is a fibration without any exceptional locus \cite[Thm\,1.5]{Reid83}, \cite[Cor.\,2.4]{Casagrande};
        \item Batyrev's result \cite[prop.\,4.1]{Batyrev91}, describing $X$ as a PTB over $B$, when is proved that the primitive collection $\pc$ is disjoint from any further primitive collection of $\Si$: this is equivalent to the maxbord condition for $\g_{\Si}$ by \cite[Prop.\,3.25]{RT-Qfproj};
        \item our result \cite[Cor.~3.23]{RT-Qfproj}, which in addition gives a complete description of the PTB $B$ starting from the weight matrix $Q$.
      \end{itemize}
  \item $\g_{\Si}$ is intbord and not maxbord w.r.t. $H_P$; then $\dim(\g_{\Si}\cap H_P)=1$ and conditions ($i$) and ($ii$) in Definition~\ref{def:bordering}.(3) occur; the plane $H'$ cuts out a facet of $\g_{\Si}=\Nef(X)$\,: dualizing the description of $NE(X)$ given by D.~Cox and C.~von Renesse \cite[Propositions~1.9 and 1.10]{Cox-vRenesse}, $H'$ turns out to be the support of a primitive collection $\pc'=\{V_{P'}\}$; the numerical class $\n'$ of $\pc'$ defines a contractible class giving rise to a contraction $\phi:X\to X'$ such that
      $$\Exc(\phi)=\bigcap_{j\,:\,\n'\cdot\q_j<0}D_j\quad\text{where}\ D_j=\overline{O(\langle\v_j\rangle)}\in\Weil(X)$$
      (see \cite[Thm.\,2.2]{Casagrande}); in particular the pull back $\phi^*(\Nef(X'))=H'\cap\g_{\Si}$ meaning that $X'$ has rank $r-1=2$; there follows two possibilities:
      \begin{itemize}
        \item[(a)] $\phi$ is a blow up, hence $\Exc(\phi)$ is a divisor and $X'$ is a smooth toric variety of Picard number $r'=2$: an example of this situation is given by $\g_2$ in Example~\ref{ex:1} and Fig.\ref{fig5}; applying part (2) of Theorem~\ref{thm:intro}, $X'$ turns out to be a PTB over a projective space, giving a fibration $$f:X'\to B\cong\P^{m}$$
      with $m<n$; the composition $\vf=f\circ\phi$ gives the fibrational contraction $\vf$ claimed in part (4) of Theorem~\ref{thm:intro}; notice that this case can occur if and only if $|\{j\,:\,\n'\cdot\q_j<0\}|=1$, which is, recalling (\ref{Mov}), if and only if $H'$ cuts out a facet of $\Mov(X)$;
      \begin{figure}
\begin{center}
\includegraphics[width=8truecm]{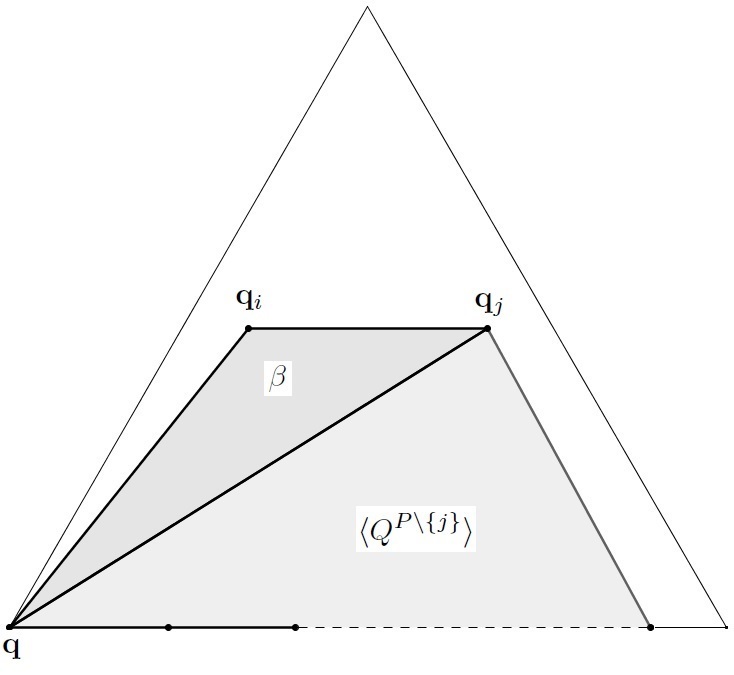}
\caption{\label{fig4}$\g_{\Si}$ no maxbord and no intbord cannot occur when $\g_{\Si}$ is a regular chamber}
\end{center}
\end{figure}
        \item[(b)] $\phi$ is a small contraction and $H'$ cuts out an internal wall of $\Mov(X)$, separating the two chambers $\g_{\Si},\g_{\widetilde{\Si}}$: therefore $X(\Si)$ and $\widetilde{X}(\widetilde{\Si})$ are birational toric varieties isomorphic in codimension 1 i.e. they are connected by an elementary toric flip determined by an internal wall crossing \cite[(15.3.14)]{CLS}; then there exists a further hyperplane $H''$, cutting out a facet of $\g_{\widetilde{\Si}}$, such that $$H_P\cap\g_{\Si}=H'\cap H_P\cap \g_{\Si}=H''\cap H'\cap H_P\cap\g_{\Si}\,;$$
            up to apply a finite number of such elementary flips we can assume that $H''$ cuts out a facet of $\Mov(X)$: Example~\ref{ex:2} and Fig.\,\ref{fig5b} give an instance of this situation; we have then two possibilities
             \begin{itemize}
               \item either $H''$ cuts out a facet of $\gkz$, which is $\g_{\widetilde{\Si}}$ is maxbord and $\widetilde{X}$ is a weighted projective toric bundle (WPTB) $\widetilde{X}\to B$, as in \cite[\S~3.3]{RT-Qfproj}: in this case the base $B$ is $\Q$-factorial with rank 2 and it is still a WPTB $B\to \P^m$ over a projective space,
               \item or $\widetilde{X}$ is a blow up of a WPTB $X''\to B\cong\P^m$, since $X''$ has rank $r-1=2$;
             \end{itemize}
             in any case, $H_P\cap \g_{\Si}$ is the pull back $\vf^*(\Nef(\P^m))$ by a fiber type contraction morphism $\vf$ (see considerations following Prop.\,2.5 in \cite{Casagrande13}), so proving part (5) of Theorem~\ref{thm:intro};
      \end{itemize}
to finish the proof we have to notice that $\Exc(\phi)$ cannot be trivial, in fact
      \begin{itemize}
        \item if $|\{j\,:\,\n'\cdot\q_j<0\}|=0$ then $\g_{\Si}$ would be maxbord w.r.t. $H'$ and we can apply the previous argument (1), by replacing $H_P$ with $H'$, and obtaining Theorem~\ref{thm:intro}.(3),

        \item Proposition~\ref{prop:primitive} gives $|\pc'|\geq 2$ and at least $r-1=2$ columns of $Q$ determines $H'$, giving
        $$|\{j\,:\,\n'\cdot\q_j<0\}|\leq n+r-2-2=n-1\ \Rightarrow\ \dim(\Exc(\phi))\geq 1\,;$$
      \end{itemize}
      finally notice that last case (b) cannot occur when $n\leq 3$: in fact, in this case $\g_{\Si}$ is an internal chamber of $\Mov(X)$, intbord and not maxbord w.r.t. $H_P$ and such that every plane $H'$, cutting a facet of $\g_\Si$ and verifying conditions ($i$) and ($ii$) in Definition~\ref{def:bordering}.(3), cannot cut out a facet of $\Mov(X)$, otherwise we would be in case (a): the situation described in Fig.\,\ref{fig5b} may be useful to fixing ideas; then there should exist at least 7 columns of the weight matrix $Q$,  namely given by:
      \begin{itemize}
        \item $\langle\q_3\rangle = H_P\cap\g_{\Si}$,
        \item two columns $\q_6,\q_7$ determining the two facet of $\g_{\Si}$ meeting in $\q_3$; that is $\langle\q_3,\q_6\rangle$ and $\langle\q_3,\q_7\rangle$ are facets of $\g_{\Si}$ cut out by planes $H'$ and $H''$, respectively, admitting inward normal vectors $\n'$ and $\n''$, respectively,
        \item two further columns $\q_2,\q_4$ determining two facets of $\Mov(X)$; that is $\langle\q_2,\q_3\rangle$ and $\langle\q_3,\q_4\rangle$ are facets of $\Mov(X)$\,: notice that these two facets have to be necessarily distinct from the previous facets of $\g_\Si=\Nef(X)$, otherwise we would be in case (a), getting a divisorial contraction determined by either $\n'$ or $\n''$,
        \item finally, two further columns $\q_1,\q_5$ determining the facet $\langle\q_1,\q_5\rangle$ of $\gkz=\Eff(X)$\,, cut out by $H_P$\,, and verifying the intbord condition ($ii$) in Definition~\ref{def:bordering}.(3), for $\g_{\Si}$\,; notice that $\q_1,\q_2,\q_4,\q_5$ have to give distinct columns of $Q$ because $\n'$ and $\n''$ have to satisfy the following numerical conditions
            $$|\{j\,:\,\n'\cdot\q_j<0\}|\geq 2\quad,\quad|\{j\,:\,\n''\cdot\q_j<0\}|\geq 2$$
            otherwise we would have a divisorial contraction, as in case (a);
      \end{itemize}
     but
      $$ 7\leq n+r =n+3\ \Rightarrow\ n\geq 4\,;$$
  \item $\g_{\Si}$ is no intbord w.r.t $H_P$; since Theorem \ref{thm:} implies that $\g_{\Si}$ is bordering w.r.t. $H_P$, we are necessarily in the case represented in Fig.\,\ref{fig4}, which cannot occur since $\pc=\{V_P\}$ is a primitive collection. In fact, if $\b=\langle\q,\q_i,\q_j\rangle\in\B(\g_{\Si})$ with $\q_i,\q_j\in\pc^*$, $\q$ be a column of $Q$ belonging to the intersection of $H_P$ with a further facet of $\gkz$, then Proposition~\ref{prop:primitive}.($ii.4$) implies that $\g_{\Si}\subseteq \langle Q^{P\setminus\{j\}}\rangle$, contradicting the fact that $\g_{\Si}\subseteq\b$ (see Fig.\,\ref{fig4}).
\end{enumerate}

\subsection{Examples}\label{ssez:ex}
\begin{example}\label{ex:1}Let us here consider and continue an example partly discussed in \cite[Ex.\,3.40]{RT-Qfproj}. Consider the following reduced $F$ and $W$ matrices, $V$ and $Q$ respectively,
\begin{figure}
\begin{center}
\includegraphics[width=7truecm]{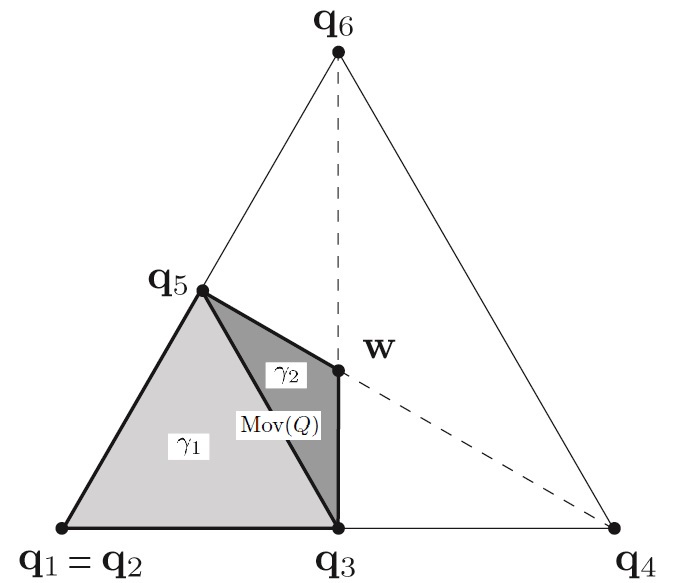}
\caption{\label{fig5}The section of the cone $\Mov(Q)$ and its chambers, inside the Gale dual cone $\mathcal{Q}=F^3_+$, as cut out by the plane $\sum_{i=1}^3x_i^2=1$, when $V$ is the fan matrix given in Example~\ref{ex:1}}
\end{center}
\end{figure}
\begin{equation*}
    V=\left(
        \begin{array}{cccccc}
          1 & 0 & 0 & 0 & -1 & 1 \\
          0 & 1 & 0 & 0 & -1 & 1 \\
          0 & 0 & 1 & -1 & -1 & 1 \\
        \end{array}
      \right)\ \Longrightarrow\ Q=\left(
                             \begin{array}{cccccc}
                               1 & 1 & 1 & 0 & 1 & 0 \\
                               0 & 0 & 1 & 1 & 0 & 0 \\
                               0 & 0 & 0 & 0 & 1 & 1 \\
                             \end{array}
                           \right)=\G(V)\,.
\end{equation*}
One can visualize the Gale dual cone $\gkz=\langle Q\rangle =F^3_+$ and $\Mov(Q)\subseteq\gkz$ as in Fig.~\ref{fig5}. Recalling
\cite{Kleinschmidt-Sturmfels}, $\P\SF(V)=\SF(V)$ and the latter is described by the only two
chambers $\g_1,\g_2$ of $\Mov(Q)$ represented in Fig.\,\ref{fig5}. We get then two $\Q$--factorial projective toric varieties $X_i=X_i(\Si_{\g_i})\,,\ i=1,2\,,$ of dimension and rank 3: both $X_1$ and  $X_2$ are smooth. Since $r=3$, we are either in case (3) or in case (4) of Theorem~\ref{thm:intro}.

For $X_1$ we are in case (3), being $\g_1$ maxbord with respect to both the hyperplanes $H_2:x_2=0$ and $H_3:x_3=0$. This fact can also be deduced by the weight matrix $Q$, which is essentially invariant (i.e. up to reordering the columns to still get a REF) with respect to the exchange of the second and the third rows. We get then two fibrations $f':X_1\twoheadrightarrow B'$ and $f'':X_1\twoheadrightarrow B''$, over smooth toric surfaces of rank 2, $B', B''$, both geometrically given by the blow up of $\P^2$ in one point, $B'\cong\P\left(\mathcal{O}_{\P^1}\oplus\mathcal{O}_{\P^1}(1)\right)\cong B''$. This fact allows us to exhibit $X_1$ as a sequence of two consecutive PTB in twofold way:
$$\xymatrix{X_1\cong\P\left(\mathcal{O}_{B'}\oplus\mathcal{O}_{B'}(h)\right)\ar@{>>}[r]& B'\cong \P\left(\mathcal{O}_{\P^1}\oplus\mathcal{O}_{\P^1}(1)\right)\ar@{>>}[r] &\P^1}$$
$$\xymatrix{X_1\cong\P\left(\mathcal{O}_{B''}\oplus\mathcal{O}_{B''}(h)\right)\ar@{>>}[r]& B''\cong \P\left(\mathcal{O}_{\P^1}\oplus\mathcal{O}_{\P^1}(1)\right)\ar@{>>}[r] &\P^1}\,.$$
See \cite[Ex.\,3.40]{RT-Qfproj} for further details.

For $X_2$ we are in case (4) of Theorem~\ref{thm:intro}, being $\g_2$ intbord and no maxbord with respect to both the hyperplanes $H_2$ and $H_3$: we still get a symmetry giving rise to two isomorphic fibrational contractions $X_2\twoheadrightarrow\P^2$. Let us consider the case of the primitive collection supported by $H_P=H_3$. Recalling notation introduced in \S\,\ref{ssez:intbord}.(2), $H'$ is given by the hyperplane passing through $\q_3$ and $\q_6$, corresponding to blowing down the divisor $D_4=\overline{O(\v_4)}$, so giving the contracting part $\phi':X_2\to B'$, over a toric projective threefold $B'$ of rank 2. This contraction can be described by performing the following operation on the fan and weight matrices:
\begin{itemize}
  \item suppress the column $\v_4$ in the matrix $V$, so getting
  \begin{equation*}
    V'=\left(
        \begin{array}{ccccc}
          1 & 0 & 0 & -1 & 1 \\
          0 & 1 & 0  & -1 & 1 \\
          0 & 0 & 1  & -1 & 1 \\
        \end{array}
      \right)
      \end{equation*}
  \item suppress the column $\q_4$ and the second row in the matrix $Q$, so getting
      \begin{equation*}
      Q'=\left(
                             \begin{array}{ccccc}
                               1 & 1 & 1 &  1 & 0 \\
                               0 & 0 & 0 &  1 & 1 \\
                             \end{array}
                           \right)=\G(V')\,.
\end{equation*}
\end{itemize}
Then $B'$ is the blow up of $\P^3$ in one point, giving the geometric description of $X_2$ as the blow up of $\P^3$ in two distinct points. The fibering part is then obtained by writing $B'\cong\P(\mathcal{O}_{\P^2}\oplus\mathcal{O}_{\P^2}(1))$, so giving $f':B'\twoheadrightarrow\P^2$, with fiber $F\cong\P^1$. We get then the fibrational contraction $f'\circ\phi':X_2\to\P^2$ with fiber $\P^1$ and exceptional locus $D_4\cong\P^2$.
\end{example}
\begin{figure}
\begin{center}
\includegraphics[width=7truecm]{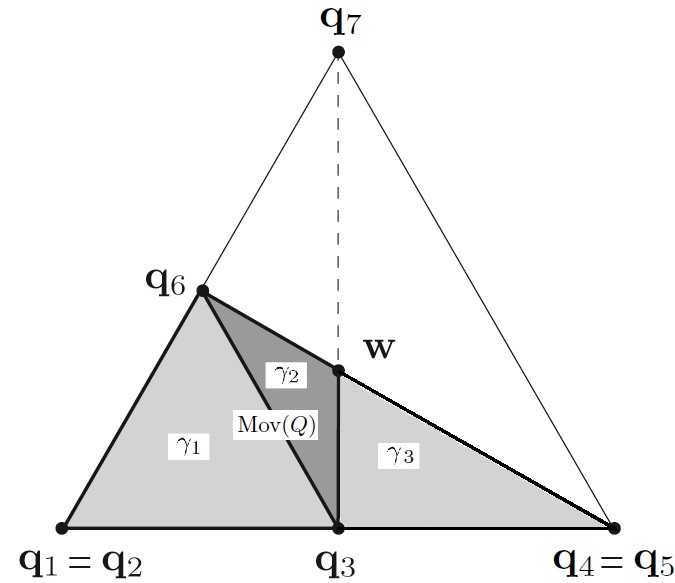}
\caption{\label{fig5b}The section of the cone $\Mov(Q)$ and its chambers, inside the Gale dual cone $\mathcal{Q}=F^3_+$, as cut out by the plane $\sum_{i=1}^3x_i^2=1$, when $V$ is the fan matrix given in Example~\ref{ex:2}}
\end{center}
\end{figure}
\begin{example}\label{ex:2}
Let us add a further generator to the weight matrix of the previous example, so getting
\begin{equation*}
  Q=\left(
                             \begin{array}{ccccccc}
                               1 & 1 & 1 & 0 & 0 & 1 & 0 \\
                               0 & 0 & 1 & 1 & 1 & 0 & 0 \\
                               0 & 0 & 0 & 0 & 0 & 1 & 1 \\
                             \end{array}
                           \right)\ \Longrightarrow\ V=\G(Q)=\left(
        \begin{array}{ccccccc}
          1 & 0 & 0 & 0 & 0 & -1 & 1 \\
          0 & 1 & 0 & 0 & 0 & -1 & 1 \\
          0 & 0 & 1 & 0 & -1 & -1 & 1 \\
          0 & 0 & 0 & 1 & -1 & 0 & 0 \\
        \end{array}
      \right)
\end{equation*}
The associated Gale dual cone $\gkz=\langle Q\rangle =F^3_+$ and $\Mov(Q)\subseteq\gkz$ are represented in Fig.~\ref{fig5b}. Chambers $\g_1$ and $\g_3$ are totally maxbord and maxbord, respectively, so giving the following sequences of two consecutive PTB
\begin{eqnarray}\label{fibrazioni}
\nonumber
   &\xymatrix{X_1\cong\P\left(\mathcal{O}_{B_1}\oplus\mathcal{O}_{B_1}(h)\right)\ar@{>>}[r]& B_1\cong \P\left(\mathcal{O}_{\P^1}\oplus\mathcal{O}_{\P^1}(1)^{\oplus 2}\right)\ar@{>>}[r] &\P^1}& \\
  &\xymatrix{X_1\cong\P\left(\mathcal{O}_{B_2}\oplus\mathcal{O}_{B_2}(h)^{\oplus 2}\right)\ar@{>>}[r]& B_2\cong \P\left(\mathcal{O}_{\P^1}\oplus\mathcal{O}_{\P^1}(1)\right)\ar@{>>}[r] &\P^1}& \\
  \nonumber
   &\xymatrix{X_3\cong\P\left(\mathcal{O}_{B_3}\oplus\mathcal{O}_{B_3}(h)\right)\ar@{>>}[r]& B_3\cong \P\left(\mathcal{O}_{\P^1}\oplus\mathcal{O}_{\P^1}(1)^{\oplus 2}\right)\ar@{>>}[r] &\P^1}\,.&
\end{eqnarray}
Let us then focus on the case of chamber $\g_2$: we get a fiber type contraction $\vf:X_2\to\P^1$ such that $\langle\q_3\rangle=\vf^*(\Nef(\P^1))$ where $\vf$ can be obtained in two ways:
\begin{itemize}
  \item starting from the first fibration $X_1\to\P^1$ in (\ref{fibrazioni}), whose fibre is locally isomorphic to $\P^1\times\P^2$, contract a fibre $\P^1$ of the fibration $X_1\to B_1$ and resolve the obtained singularity by a small blow up inserting an exceptional $E\cong\P^2$: this gives a fiber type contraction morphism $X_2\to\P^1$ which is given by the fibration $X_1\to\P^1$ in which one fibre is replaced by $E\times\P^2$;
  \item the same fibre type contraction $\vf$ can be symmetrically obtained from the third fibration $X_3\to\P^1$ in (\ref{fibrazioni}), by proceeding in the same way.
\end{itemize}
\end{example}

\section{What about $r\geq 4$\,?}\label{sez:r>3}

This section is meant to study any possible extension of Theorem~\ref{thm:} to higher values of the Picard number $r\geq~4$.

Let us first of all observe that, on the one hand, Theorem \ref{thm:} clearly holds in dimension 2 for any rank: in fact every smooth, complete, toric surface $X$ is projective and if $X\not\cong\P^2$ then there exists a morphism $f:X\to\P^1$, so that $f^*\mathcal{O}_{\P^1}(1)$ is nef and non-big.

On the other hand, let us recall that, in 2009, Fujino and Sato \cite{FS} produced a series of (counter-)examples, for every dimension $n\geq 3$ and every rank $r\geq 5$, of smooth, projective, toric varieties whose $\Nef$ cone is not bordering i.e. their nef divisors are all big divisors. Moreover in the case of complete and non-projective, smooth, toric varieties, the classical Oda's example \cite[Prop.~9.4]{OM},\cite[p.~84]{Oda} exhibits a  complete, non-projective, toric threefold (i.e. smooth and 3-dimensional variety) of rank 4, whose $\Nef$ cone is 2-dimensional and no bordering. This example can be easily generalized in grater dimension by augmenting the multiplicity of suitable columns inside a weight matrix, as done in \ref{ssez:generalizzazione} for the example \ref{ssez:esempio}.

Therefore it remains to understand what happens for smooth, projective, toric varieties of dimension $n\geq 3$ and Picard number $r=4$.

On the one hand, the Oda-Miyake result \cite[Thm.~9.6]{OM} allows us to conclude that Theorem \ref{thm:} can be extended to the case of dimension $n=3$ and rank $r=4$: in \S~\ref{ssez:OM} we will consider in more detail this case, leading us to a classification of all threefolds with Picard number $r\leq 4$ (see the following Theorem~\ref{thm:n=3,r=4}).

On the other hand, for $n\geq 4$ there is no hope of getting a similar statement, due to counterexamples \ref{ssez:esempio} and \ref{ssez:generalizzazione}. Let us point out that it is quite hard to get these examples by means of purely geometric considerations, as also Fujino and Sato observed \cite[pg.~1]{FS}. We discovered them with the help of Maple routines which enabled us to treating the combinatorics of the primary and the secondary fans. Unfortunately, we were not yet able to find time and willing of putting those routines in a sufficiently user-friendly shape: it is a work in progress. Anyway we can provide the raw Maple spreadsheet to anyone who was interested in.

The only property which can be arbitrarily generalized on the rank $r$ is given by Lemma \ref{lm:bordante}. In fact it holds for every $r\geq 1$, provided the existence of a primitive nef (hence bordering) collection $\pc=\{V_P\}$ for a simplicial  fan $\Si$: the latter is guaranteed for a regular and projective fan by Batyrev' results \cite[Prop.~3.2]{Batyrev91}. Here is a proof.

\begin{figure}
\begin{center}
\includegraphics[width=8truecm]{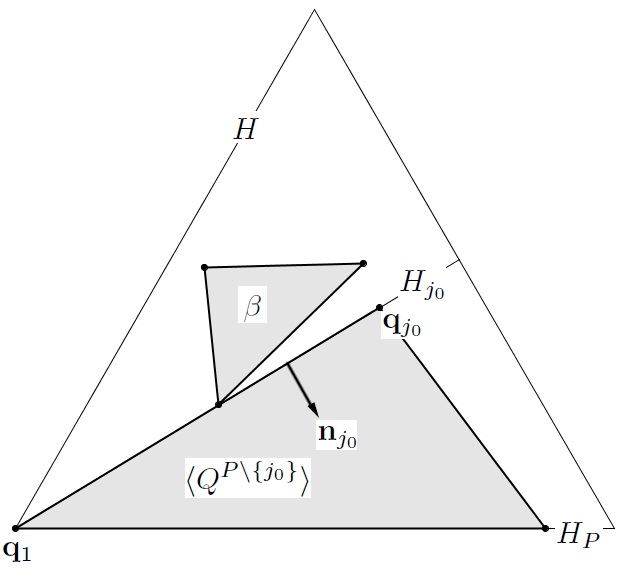}
\caption{\label{fig6}The argument proving Lemma \ref{lm:bordante} when $r=3$}
\end{center}
\end{figure}

\subsection{Proof of Lemma \ref{lm:bordante}}
  The support hyperplane $H_P\subset F^r_{\R}\cong\R^r$, of the primitive collection $\pc$, cut out a facet of $\gkz=\langle Q\rangle=\Eff(X(\Si))$. Let $H$ be a hyperplane cutting out an adjacent facet of $\gkz$. Then $H\cap H_P$ contains $r-2$ linearly independent   columns $\q_1,\ldots,\q_{r-2}$ of the weight matrix $Q$. For every $\q_j \in \pc^*$ let $H_j$ be the hyperplane passing through $\q_1,\ldots,\q_{r-2}$ and $\q_j$: then $H_j$ cuts out a facet of the cone $\langle Q^{P\setminus\{j \}}\rangle $, which contains $\gamma_\Sigma$ by Prpoposition~\ref{prop:primitive}.($ii.4$).  Let $\n_j$ be a non zero normal vector  to $H_j$ such that
  \begin{equation}\label{eq:qq1q}
 \n_j\cdot x\geq 0 \hbox{ for every } x\in\g_\Sigma.\end{equation}
  There certainly exists an index $j_0$ such that $\q_{j_0}\in\pc^*$ and $\n_{j_0}\cdot \q_j\leq 0$ for every $\q_j\in\pc^*$. By absurd let $\beta\in\B(\g_\Si)$ be such that $\beta(1)\cap\{Q^P\}=\emptyset$. Then $\n_{j_0}\cdot \pp\leq 0$ for every $\pp\in\beta(1)$. Since $\gamma$ is $r$-dimensional $\n_{j_0}\cdot x <0$ for every $x$ in the interior of $\gamma_\Sigma$, against (\ref{eq:qq1q}) (see Fig.\,\ref{fig6} for a description of  this situation when $r=3$).

\subsection{Toric projective threefolds of Picard number 4}\label{ssez:OM}

In the following \emph{threefold} means a smooth and 3-dimensional variety. As a consequence of \cite[Thm.~9.6]{OM} one gets the following
\begin{theorem}\label{thm:OM}
The statement of Theorem~\ref{thm:} can be extended to the case of toric, projective threefolds of Picard number $r=4$.
\end{theorem}

\begin{proof}  Let $X(\Si)$ be a projective, smooth, toric, variety of dimension $n=3$ and rank $r=4$. If $X$ is not the blow up of a smooth toric variety of lower rank, then \cite[Thm.~9.6]{OM} ensures the existence of a smooth, projective, toric surface $B$ of rank 3 and such that $X\cong\P(\mathcal{O}_B\oplus\mathcal{O}_B(D))$, for some $\mathcal{O}_B(D)\in\Pic(B)$. This means that, calling $f:X\twoheadrightarrow B$ the canonical fibration,
\begin{equation*}
  \dim(\g_{\Si}\cap\partial\gkz)=\dim(f^*(\Nef(B)))=3\,.
\end{equation*}
Therefore $\g_{\Si}$ is maxbord w.r.t. a hyperplane $H$, which turns out to be the support of the primitive nef collection $\pc=\{V_P\}$, where $P\subset\{1,\ldots,n+r\}$ counts the columns of $Q$ not belonging to $H$.

\noindent We can then assume $X\to X'(\Si')$ to be the blow up of a projective, toric threefold $X'$ of rank 3. Then Theorem~\ref{thm:} applies to $X'$ exhibiting a nef, primitive collection $\pc$ of $\Si'$ such that $\Nef(X')=\g_{\Si'}$ is bordering w.r.t. the support hyperplane $H'_P$ of $\pc$ i.e. $\dim(\g_{\Si'}\cap H'_P)\geq 1$. Notice that $|\pc|\leq 3$ otherwise one can easily obtain $X'\cong\P^3$, against $\rk(X')=3$.  Let $V'\in\M(3,n+3;\Z)$ be a fan matrix of $X'$. Then a fan matrix of $X$ is given by $V=(V'|\v)$, where $\langle\v\rangle$ is the ray associated with the exceptional divisor $E$ of the blow up. In particular for every cone $\s\in\Si$ either $\s\in\Si'$ or $\langle\v\rangle\subseteq\s\subsetneq\s'\in\Si'$. This is enough to show that $\pc$ can't be contained in any cone of the fan $\Si$. Moreover, for every $j\in P$, consider a cone $\s'\in\Si'$ such that $\pc\backslash\{\v_j\}\subset\s'$. Then either $\s'\in\Si$ or there exists $\s\in\Si$ such that $\pc\backslash\{\v_j\}\cup\{\v\}\subset\s$, since $|\pc\backslash\{\v_j\}\cup\{\v\}|\leq 3$. Then $\pc$ turns out to be a primitive nef collection of $\Si$, too. Let $H_P$ be its support hyperplane in $\Cl(X)\otimes\R\cong\R^4$: clearly $H_P$ cuts out a facet of $\gkz$ since $\pc$ is nef. Moreover $\dim(\g_{\Si}\cap H_P)=\dim(\g_{\Si'}\cap H'_P)\geq 1$, showing that $\g_{\Si}$ is bordering w.r.t. $H_P$.
\end{proof}

The previous Theorem \ref{thm:OM} and Theorem \ref{thm:intro} allow us to obtain the following classification of threefolds with Picard number $r\leq 4$.
\begin{theorem}\label{thm:n=3,r=4} Let $X(\Si)$ be a toric, projective threefold of Picard number $2\leq r\leq 4$. Then there exists a fiber type contraction $\vf:X\twoheadrightarrow B$ over a smooth projective toric surface $B$, whose Picard number $r_B\leq 3$, such that either $B\cong\P^2$ or $B$ is (possibly a blow up of) a PTB over $\P^1$.
  In particular we get the following:\\
  \textbf{\emph{Classification:}}
\begin{enumerate}
   \item $r=1$ and $X\cong\P^3$;
    \item $r=2$ and $\vf$ is a fibration exhibiting $X$ as a PTB over $\P^2$;
    \item $r=3$ and $\vf:X\to B$ is a fibration exhibiting $X$ as a PTB over a smooth toric surface $B$ of Picard number 2, which is still a PTB over $\P^1$, meaning that $X$ is obtained from $\P^1$ by a sequence of two projectivizations of decomposable toric bundles;
    \item $r=3$ and $\vf=f\circ\phi$ is a fibrational contraction over $\P^2$ such that $\phi:X\to X'$ is a blow up and $f$ exhibits $X'$ as a PTB, $X'\cong\P(\mathcal{O}_{\P^2}\oplus\mathcal{O}_{\P^2}(a))$, for some $a\geq 0$;
  \item $r=4$ and $\vf:X\to B$ is a fibration exhibiting $X$ as a PTB over a blow up $B$ of a Hirzebruch surface $\P(\mathcal{O}_{\P^1}\oplus\mathcal{O}_{\P^1}(a))$, for some $a\geq 0$,
  \item $r=4$ and $\vf=f\circ\phi:X\to B$ is a fibrational contraction over a smooth toric surface of Picard number 2, such that $\phi:X\to X'$ is a blow up and $f:X'\to B$ is a PTB over  $B\cong\P(\mathcal{O}_{\P^1}\oplus\mathcal{O}_{\P^1}(a))$, for some $a\geq 0$;
  \item $r=4$ and $\vf=f\circ\phi$ is a fibrational contraction over $\P^m$, with $m=1,2$, such that $\phi=\phi''\circ\phi'$ is a double blow up of a toric threefold $X''$ and $f$ exhibits $X''$ as a PTB over $\P^m$.
\end{enumerate}
\end{theorem}
\begin{proof} Cases from (1) to (4) are precisely the 3-dimensional occurrence of the same cases in Theorem~\ref{thm:intro}.

Let us then assume $r=4$. Theorem \ref{thm:OM} ensures that the chamber $\g_{\Si}$ can be assumed to be bordering w.r.t. the support hyperplane $H_P$ of a nef (hence bordering) primitive collection $\pc=\{V_P\}$. The hyperplane $H_P$ cuts out a facet of $\gkz=\langle Q\rangle$: in particular $1\leq\dim(\Nef(X)\cap H_P)\leq 3$.

On the one hand, if $\dim(\Nef(X)\cap H_P)=3$ (i.e. $\g_{\Si}$ is maxbord) then $X$ is a PTB, $X\cong\P(\mathcal{O}_B\oplus\mathcal{O}_B(D))$, where $B$ is a smooth toric surface of rank 3 and $\mathcal{O}_B(D)\in\Pic(B)$. In particular $B$ turns out to be the blow up of a Hirzebruch surface $\P(\mathcal{O}_{\P^1}\oplus\mathcal{O}_{\P^1}(a))$, for some $a\geq0$, so giving part (5) in the statement.

On the other hand, if $\dim(\Nef(X)\cap H_P)\leq 2$ then \cite[Thm.~9.6]{OM} implies that $X$ has necessarily to be a blow up of a toric threefold $X'$ of rank $r'=3$ and Theorem~\ref{thm:intro} applies to $X'$. Since $\dim(X')=3$ we have only to consider cases (3) and (4) in Thm.~\ref{thm:intro}. The former gives part (6) in the statement and corresponds to having $\dim(\Nef(X)\cap H_P)=2$. The latter gives part (7) in the statement and corresponds to having $\dim(\Nef(X)\cap H_P)=1$.
\end{proof}

\subsection{A 4-dimensional counterexample with Picard number 4}\label{ssez:esempio}

Consider the following reduced $F$ and $W$ matrices, $V$ and $Q$ respectively,
\begin{eqnarray*}
  V &=& \left(
          \begin{array}{cccccccc}
            1&0&0&-1&0&1&-1&0 \\
            0&1&0&1&0&0&-1&1 \\
            0&0&1&1&0&-1&0&1 \\
            0&0&0&0&1&-1&1&0 \\
          \end{array}
        \right)
   \\
  Q=\G(V) &=& \left(
                \begin{array}{cccccccc}
                  1&0&0&1&0&1&1&0 \\
                  0&1&1&0&0&1&1&0 \\
                  0&0&1&1&1&2&1&0 \\
                  0&0&0&0&0&1&1&1 \\
                \end{array}
              \right)
\end{eqnarray*}
It turns out that $|\SF(V)|=10$. All these fans give rise to projective, $\Q$-factorial toric varieties of dimension and rank $n=4=r$: 8 of them are regular fans, giving rise to as many smooth projective toric varieties.

Given a fan $\Si_i\in\SF(V)$, $1\leq i\leq 10$, let $\g_i=\Nef(X(\Si_i))$ be the associated chamber of the secondary fan. Then we get
\begin{eqnarray}\label{camere}
  \g_1 &=& \langle\q_3,\q_4,\w_1,\w_4\rangle \quad\text{smooth}\\
  \nonumber
  \g_2 &=& \langle\q_4,\w_1,\w_2,\w_4\rangle \quad\text{smooth}\\
  \nonumber
  \g_3 &=& \langle\q_3,\w_1,\w_3,\w_4\rangle \quad\text{smooth}\\
  \nonumber
  \g_4 &=& \langle\q_7,\w_2,\w_3,\w_5\rangle \quad\text{singular}\\
  \nonumber
  \g_5 &=& \langle\q_3,\q_4,\q_6,\w_4\rangle \quad\text{smooth}\\
  \nonumber
  \g_6 &=& \langle\q_6,\q_7,\w_2,\w_3\rangle \quad\text{singular}\\
  \nonumber
  \g_7 &=& \langle\q_3,\q_6,\w_3,\w_4\rangle \quad\text{smooth}\\
  \nonumber
  \g_8 &=& \langle\w_1,\w_2,\w_3,\w_4,\w_5\rangle \quad\text{smooth}\\
  \nonumber
  \g_9 &=& \langle\q_4,\q_6,\w_2,\w_4\rangle \quad\text{smooth}\\
  \nonumber
  \g_{10} &=& \langle\q_6,\w_2,\w_3,\w_4\rangle \quad\text{smooth}
\end{eqnarray}
where $Q=(\q_1,\ldots,\q_8)$ and
$$
\w_1=\left(
       \begin{array}{c}
         1 \\
         1 \\
         1 \\
         0 \\
       \end{array}
     \right)\ ,\
\w_2=\left(
       \begin{array}{c}
         2 \\
         1 \\
         2 \\
         1 \\
       \end{array}
     \right)\ ,\
\w_3=\left(
       \begin{array}{c}
         1 \\
         2 \\
         2 \\
         1 \\
       \end{array}
     \right)\ ,\
\w_4=\left(
       \begin{array}{c}
         2 \\
         2 \\
         3 \\
         1 \\
       \end{array}
     \right)\ ,\
\w_5=\left(
       \begin{array}{c}
         2 \\
         2 \\
         2 \\
         1 \\
       \end{array}
     \right)
$$
The pseudo-effective cone $\Eff(X(\Si_i))=\langle Q\rangle=\langle\q_1,\q_2,\q_5,\q_8\rangle=F^4_+$ is given by the positive orthant of $F^4_+\subset F^4_{\R}=\Cl(X(\Si_i))\otimes \R\cong\R^4$: then a chamber $\g_i$ is bordering if and only if it admits a generator with some 0 entry. Therefore $\g_{10}$ is no-bordering, meaning that $X=X(\Si_{10})$ is a smooth, projective, toric variety whose nef cone $\Nef(X)$ is internal to $\Eff(X)$: then every nef divisor of $X$ is big.

\noindent Maximal cones of the fan $\Si_{10}$ are the following
\begin{eqnarray*}
\Si_{10}(4)=&\{\langle2, 4, 5, 7\rangle, \langle4, 5, 7, 8\rangle, \langle3, 4, 7, 8\rangle, \langle3, 4, 6, 7\rangle, \langle2, 4, 6, 7\rangle, \langle3, 5, 7, 8\rangle,&\\ &\langle2, 4, 5, 8\rangle, \langle1, 3, 5, 7\rangle, \langle2, 5, 6, 7\rangle, \langle3, 4, 6, 8\rangle, \langle2, 4, 6, 8\rangle, \langle1, 5, 6, 7\rangle,&\\ &\langle1, 3, 6, 7\rangle, \langle1, 3, 5, 8\rangle, \langle1, 2, 5, 8\rangle, \langle1, 2, 5, 6\rangle, \langle1, 3, 6, 8\rangle, \langle1, 2, 6, 8\rangle\}&
\end{eqnarray*}
where we adopted the notation $\langle i,j,k,l\rangle:=\langle \v_i,\v_j,\v_k,\v_l\rangle$.

A geometric description of $X$ can be obtained by observing that $\g_1$ is a maxbord chamber w.r.t. $H_4:x_4=0$, where $x_1,x_2,x_3,x_4$ are coordinates of $F^4_{\R}\cong\R^4$: in fact $\q_3,\q_4,\w_1\in H_4$. Then, following methods described in \cite[\S~3]{RT-Qfproj}, one observes that
$$
Q\sim \left(
        \begin{array}{cccccccc}
          1&0&0&1&0&0&0&-1 \\
          0&1&1&0&0&0&0&-1 \\
          0&0&1&1&1&0&-1&-2 \\
          0&0&0&0&0&1&1&1 \\
        \end{array}
      \right)
$$
Calling $X_1:=X_1(\Si_1)$, the fact that $\g_1$ is maxbord w.r.t. $H_4$ implies that the bottom row of $Q$ gives a primitive relation of $\Si_1$ and $X_1$ is a PTB over a smooth toric surface $B$ of Picard number $3$. Namely, the submatrix $Q'$ of $Q$, obtained by deleting the bottom row and the columns $\q_6,\q_7,\q_8$, giving the primitive collection of $\Si_1$ supported by $H_4$, is
$$Q'=\left(
       \begin{array}{ccccc}
         1 & 0 & 0 & 1 & 0 \\
         0 & 1 & 1 & 0 & 0 \\
         0 & 0 & 1 & 1 & 1 \\
       \end{array}
     \right)\ \Rightarrow\ V'=\G(Q')=\left(
              \begin{array}{ccccc}
                1&0&0&-1&1 \\
                0&1&-1&0&1 \\
              \end{array}
            \right)
$$
There exists a unique fan $\Si'\in\SF(V')$ and $B=B(\Si')$ turns out to be a blow up in one point of $\P^1\times\P^1$. Let $e$ be the exceptional $\P^1$ of the blow up and $h_1,h_2$ be the strict transforms of the generators of $\Pic(\P^1\times\P^1)$. Then
$$X_1\cong\P\left(\mathcal{O}_B\oplus\mathcal{O}_B(e)\oplus\mathcal{O}_B(h_1+h_2+2e)\right)\ .$$
By observing that the pseudo-effective cone $\langle Q\rangle$ is symmetric w.r.t. the hyperplane of $F^4_{\R}$ passing through $\q_5,\q_7,\q_8$ and the origin, we obtain at least two isomorphic ways of describing the toric flip $X_1\dashrightarrow X$, which is an isomorphism in codimension 1. In any case, it is realized by at least three consecutive \emph{wall-crossings} whose wall relations are given by multiplying the respective inward normal vectors against the columns of $Q$. The two minimal isomorphic sequences of flips follow by the following chambers' adjacencies
\begin{equation}\label{flips}
\xymatrix{&\g_7\ar@{<~>}[rr]^-{\langle\q_3,\w_3,\w_4\rangle}_-{\n_7}&&\g_3
   \ar@{<~>}[dr]^-{\langle\q_3,\w_1,\w_4\rangle}_-{\n_3}&\\
   \g_{10}\ar@{<~>}[ur]^-{\langle\q_6,\w_3,\w_4\rangle}_-{\n_{10}}\ar@{<~>}[dr]_-{\langle\q_6,\w_2,\w_4\rangle}&&&&\g_1\\
   &\g_9\ar@{<~>}[rr]_-{\langle\q_4,\w_2,\w_4\rangle}&&\g_2
   \ar@{<~>}[ur]_-{\langle\q_4,\w_1,\w_4\rangle}&&\\}
\end{equation}
in which common facets of adjacent chambers have been spelled out. Since the involved chambers are simplicial, inward normal vectors to the facets of the chamber $\g_i$ ($i\neq 8$) are easily determined by the rows of the inverse matrix $G_i^{-1}$, where $G_i$ is the $4\times 4$ integer matrix defined by the generators of $\g_i$, as given in (\ref{camere}).

Let us choose the upper sequence in (\ref{flips}). By the following inverse matrices
\begin{equation*}
  G_{10}^{-1}=\left(
              \begin{array}{cccc}
                -1&-1&1&1 \\
                0&0&1&-2 \\
                1&0&-1&1 \\
                0&1&-1&1 \\
              \end{array}
            \right) ,\quad G_7^{-1}=\left(
                                   \begin{array}{cccc}
                                     -1&-1&1&1 \\
                                     1&0&0&-1 \\
                                     0&1&-1&1 \\
                                     -1&0&1&-1 \\
                                   \end{array}
                                 \right) ,
\end{equation*}
\begin{equation*}
G_3^{-1}=\left(
                                   \begin{array}{cccc}
                                     -1&0&1&-1 \\
                                     0&-1&1&0 \\
                                     1&1&-1&-1 \\
                                     0&1&-1&1 \\
                                   \end{array}
                                 \right)
\end{equation*}
 one can easily deduce that the inward, w.r.t. $\g_{10}$, normal vector to the common facet $\g_{10}\cap\g_7=\langle\q_3,\w_3,\w_4\rangle$ is $\n_{10}=(1,0,-1,1)$; analogously $\n_7=(-1,-1,1,1)$ is the inward, w.r.t. $\g_7$, normal vector to $\g_7\cap\g_3=\langle\q_3,\w_3,\w_4\rangle$ and $\n_3=(0,1,-1,1)$ is the inward, w.r.t. $\g_3$, normal vector to $\g_3\cap\g_1=\langle\q_3,\w_1,\w_4\rangle$. The wall relations are then given by
\begin{eqnarray*}
  \n_{10}\cdot Q &=& (1,0,-1,0,-1,0,1,1) \\
  \n_{7}\cdot Q &=& (-1,-1,0,0,1,1,0,1) \\
  \n_{3}\cdot Q &=& (0,1,0,-1,-1,0,1,1)
\end{eqnarray*}
meaning that the associated sequence of flips
$$X_1\stackrel{\varphi_3}{\dashrightarrow} X_3\stackrel{\varphi_7}{\dashrightarrow} X_7\stackrel{\varphi_{10}}{\dashrightarrow} X$$
is defined as follows:
\begin{itemize}
  \item $\varphi_3$ contracts $D_4\cap D_5$ and $\varphi_3^{-1}$ contracts $D_2\cap D_7\cap D_8$
  \item $\varphi_7$ contracts $D_1\cap D_2$ and $\varphi_7^{-1}$ contracts $D_5\cap D_6\cap D_8$
  \item $\varphi_{10}$ contracts $D_3\cap D_5$ and $\varphi_{10}^{-1}$ contracts $D_1\cap D_7\cap D_8$
\end{itemize}
where, as usual the prime divisor $D_i$ is the closure of the torus orbit $D_i=\overline{O(\v_i)}$, whose definition does not depend on the fan choice in $\SF(V)$. To give an account of what is the modification induced by the codimesion 1 isomorphism $\varphi=\varphi_{10}\circ\varphi_{7}\circ\varphi_{3}$, notice that $\Nef(B)=\langle \q_3',\q_4',\w_1'\rangle$, where $\q'_i$ are columns of $Q'$ and $\w_1'=\left(\begin{matrix}
1 \\ 1\\ 1\\\end{matrix}\right)$. Calling $\pi$ the canonical projection $X_1\to B$, then
$$\pi^*\left(\Nef(B)\right)=\langle \q_3,\q_4,\w_1\rangle\subset\Nef(X_1)=\left\langle[D_1],[D_3],[D_4]\right\rangle=\left\langle[D_2],[D_3],[D_4]\right\rangle$$
describes the locus of non-big divisors inside $\Nef(X_1)$. Then $\varphi$ modifies the intersection properties of these divisors so that they are no more nef on $X$. This fact can be directly checked by observing that $\NE(X)=\left\langle (G_{10}^{-1})^T\right\rangle$.

Recalling \cite[Thm.~1.4]{Cox-vRenesse}, the inverse matrix $G_{10}^{-1}$ is essential to determining the primitive relations generating $\Nef(X)$, given by
\begin{equation*}
  G_{10}^{-1}\cdot Q\cdot\left(
                           \begin{array}{c}
                             \v_1 \\
                             \vdots \\
                             \v_8 \\
                           \end{array}
                         \right)= \left(
                                    \begin{array}{c}
                                      -\v_1-\v_2+\v_5+\v_6+\v_8 \\
                                      \v_3+\v_4+\v_5-\v_7-2\v_8 \\
                                      \v_1-\v_3-\v_5+\v_7+\v_8 \\
                                      \v_2-\v_4-\v_5+\v_7+\v_8 \\
                                    \end{array}
                                  \right)
\end{equation*}
Notice that their sum gives the primitive relation $\v_6+\v_7+\v_8$, corresponding to the bottom row of $Q$.

Let us finally observe that the anti-canonical class in $\Cl(X)$ is given by
$$[-K_X]=Q\cdot\left(
          \begin{array}{c}
            1 \\
            \vdots \\
            1 \\
          \end{array}
        \right)= \left(
                   \begin{array}{c}
                     4 \\
                     4 \\
                     6 \\
                     3 \\
                   \end{array}
                 \right)= \q_6+\w_2+\w_3
$$
meaning that $[-K_X]\in\partial\Nef(X)$. Therefore \emph{$-K_X$ is a non-ample, nef and big divisor} i.e. $X$ is a \emph{weak Fano} toric fourfold (see also \cite[\S~3]{RT-Ample} for more details).

\subsection{Generalizing the counterexample to higher dimension}\label{ssez:generalizzazione}

Consider the W-matrix $Q$ defined in \ref{ssez:esempio}. The definition (\ref{Mov}) of $\Mov(Q)$ implies that the column $\q_7$ of $Q$ always belongs to (the boundary of ) $\Mov(Q)$. Actually it turns out that $\Mov(Q)=\langle \q_3, \q_4, \q_6, \q_7, \w_1\rangle$. Therefore augmenting the multiplicity of $\q_7$ produces positive, REF, $W$-matrices
\begin{equation*}
  Q_s=\left(
        \begin{array}{cccccc}
          1&0&0&1&0&1 \\
          0&1&1&0&0&1 \\
          0&0&1&1&1&2 \\
          0&0&0&0&0&1 \\
        \end{array}\right.\overbrace{\begin{array}{ccc}
                                1 & \cdots & 1 \\
                                1 & \cdots & 1 \\
                                1 & \cdots & 1 \\
                                1 & \cdots & 1
                              \end{array}}^{s\ \text{times}}
                              \left. \begin{array}{c}
                                       0 \\
                                       0 \\
                                       0 \\
                                       1
                                     \end{array}
      \right)
\end{equation*}
giving the same secondary fan than $Q$, i.e. $\Ga(Q_s)=\Ga(Q)$.

 \noindent But $V_s=\G(Q_s)\in\mathbf{M}(s+3,4;\Z)$ is a fan matrix of $(s+3)$-dimensional toric varieties. As above $|\SF(V)|=10$ and 8 of these fans give smooth, projective, toric varieties of dimension $n=s+3\geq 4$ and rank $r=4$. The same choice of the chamber $\g_{10}=\langle\q_6,\w_2,\w_3,\w_4\rangle$ gives a $n$-dimensional, smooth, projective, toric variety $X=X(\Si_{10})$ \emph{whose non-trivial nef divisors are big}.

 For a geometric description of $X$, let us proceed as above:
 \begin{equation*}
   Q_s\sim\left(
            \begin{array}{cccccccccc}
              1&0&0&1&0&0& 0 & \cdots & 0&-1 \\
              0&1&1&0&0&0& 0 & \cdots & 0&-1 \\
              0&0&1&1&1&0& -1 & \cdots & -1&-2 \\
              0&0&0&0&0&1& 1 & \cdots & 1&1
            \end{array}
          \right)
 \end{equation*}
 gives
 $$\xymatrix{X_1=X_1(\Si_1)\cong\P\left(\mathcal{O}_B\oplus\mathcal{O}_B(e)^{\oplus s}\oplus\mathcal{O}_B(h_1+h_2+2e)\right)\ar@{-->}[r]^-{\varphi}& X}$$
 where $B$ is still a blow up in one point of $\P^1\times\P^1$ and the same wall-crossings as above describe the equivariant birational equivalence and isomorphism in codimension 1 $\varphi=\varphi_{10}\circ\varphi_7\circ\varphi_3$. Finally one can easily notice that $[-K_X]\not\in \Nef(X)$, for any $s>1$.

 \begin{acknowledgements}
      We warmly thank Cinzia Casagrande for stimulating conversation and deep remarks originating, inter alia, the alternative proof of Theorem \ref{thm:} given in \S\,\ref{ssez:batyrev}. We are also indebted with Brian Lehmann who promptly informed us about the reference \cite{FS}, so giving a strong improvement of results in the present paper.
    \end{acknowledgements}

\end{document}